\documentclass[12pt]{amsart}
\usepackage[all]{xy}
\usepackage{amssymb}
\usepackage{mathrsfs}
\calclayout
\newtheorem{dummy}{anything}[section]
\newtheorem{theorem}[dummy]{Theorem}

\newtheorem*{thma}{Theorem A}
\newtheorem*{thmb}{Theorem B}
\newtheorem*{thmc}{Theorem C}

\newtheorem{lemma}[dummy]{Lemma}

\newtheorem{corollary}[dummy]{Corollary}

\theoremstyle{definition}
\newtheorem{definition}[dummy]{Definition}

\newtheorem{remark}[dummy]{Remark}

\newcommand{\cH}{\mathcal H}

\newcommand{\cE}{\mathcal E}

\newcommand{\Aut}{\operatorname{Aut}}
\newcommand{\Isom}{\operatorname{Isom}}

\newcommand{\id}{\mathrm{id}}

\newcommand{\connsum}{\mathbin{\sharp}}

\newcommand{\mmatrix}[4]{\left (\vcenter
{\xymatrix@C-2pc@R-2pc{#1&#2\\#3&#4} } \right )}

\DeclareMathOperator{\Hom}{Hom}

\DeclareMathOperator{\hepta}{Aut}

\newcommand{\hept}[1]{\hepta_{\bullet}(#1)}

\newcommand{\heqpt}[1]{\cE_{\bullet} (#1)}
\newcommand{\htildeM}{\widetilde \cH (M)}

\newcommand{\Om}{\Omega}
\newcommand{\rOm}{\widehat\Omega}
\newcommand{\B}{BO\langle n+1 \rangle}

\newtheorem*{thmd}{Theorem D}
\newtheorem*{thme}{Theorem E}
\newtheorem*{thmf}{Theorem F}
\begin{document}
\title[Homotopy Self-Equivalences of Manifolds]
{Homotopy Self-Equivalences of Manifolds}
\author{Mehmetc\.{i}k Pamuk}
\subjclass[2010]{Primary: 57N65; Secondary: 55P10, 57N13, 57R80}

\address{Department of Mathematics
\newline\indent
Middle East Technical University
\newline\indent
Ankara 06531, Turkey}
\email{mpamuk{@}metu.edu.tr}
\date{\today}

\begin{abstract}
Hambleton and Kreck described the group of based homotopy self-equivalences of a closed oriented $4$-manifold by a commutative braid of exact sequences built from bordism groups and the Postnikov $2$-type. We construct the analogous braid for closed oriented $BO\langle n+1\rangle$-manifolds $M$ of dimension $2n$ or $2n-1$ whose fundamental group is finite or infinite cyclic and which satisfy one condition on the action of $\pi_1$ on the homotopy of the fibre of $M\to P_n(M)$. We compute $\operatorname{Aut}_\bullet(S^1\times S^{m-1})\cong(\mathbb{Z}/2)^2$ for $m\ge4$, and $\operatorname{Aut}_\bullet(S^3\times S^3)\cong(\mathbb{Z}/12)^2\rtimes SL_2(\mathbb{Z})$ with the standard action modulo $12$. Further, $\operatorname{Aut}_\bullet(S^3\times S^4)\cong(\mathbb{Z}/2\oplus\mathbb{Z}/24)\rtimes(\mathbb{Z}/2)^2$ has order $192$, that $\operatorname{Aut}_\bullet(S^4\times S^4)\cong D_4\times\mathbb{Z}/2$ has order $16$, and that for connected sums of $r\ge2$ copies of $S^3\times S^3$ the kernel contains a subgroup $(\mathbb{Z}/12)^{2r}$ on which the permutations of the summands act non-trivially. For $S^p\times S^p$ the kernel is a sum of two copies of $\pi_{2p}(S^p)$ modulo Whitehead products with $\iota_p$, and the image is determined by an obstruction on the Whitehead square. In particular
$\operatorname{Aut}_\bullet(S^5\times S^5)$ is isomorphic to the theta subgroup of $SL_2(\mathbb{Z})$
and $\operatorname{Aut}_\bullet(S^7\times S^7)$ is an extension of $SL_2(\mathbb{Z})$ by $(\mathbb{Z}/120)^2$.
\end{abstract}

\maketitle

\section{Introduction}
\label{sec:intro}

For a closed, connected, oriented manifold $M$, the group $\Aut_\bullet(M)$ of based homotopy classes of orientation-preserving homotopy self-equivalences is related to surgery theory, to the $s$-cobordism classification of manifolds and to mapping class groups, and is hard to compute. Arkowitz \cite{arkowitz} surveys what is known for general spaces. Even for simply connected spaces no general structure theorem is available: Costoya and Viruel \cite{costoya-viruel} show that every finite group is the group of self-homotopy equivalences of infinitely many elliptic spaces, and that these spaces can be taken to be rationalizations of simply connected inflexible manifolds. Their passage from a Sullivan algebra to a manifold adjoins to the algebra a generator $y$ of degree $2n-1$ with $dy$ equal to the fundamental class \cite[Section~3]{costoya-viruel}. This is the algebraic model of an $S^{2n-1}$-bundle whose Euler class is the fundamental class. Costoya, Mu\~noz and Viruel \cite{cmv} use related constructions to realize prescribed sets of mapping degrees.

Products of spheres have been treated more concretely. Sieradski \cite{sieradski} considers the group $G(A\times B)$ of self-equivalences of a product of two connected $H$-spaces and obtains exact sequences $1\to[A\wedge B,A\times B]\to G(A\times B)\to GL_2(\Lambda)\to1$ under a vanishing hypothesis. His examples include $S^1\times S^n$ for $n=3,7$, where the group is $(\mathbb{Z}/2)^3$, and the products $S^3\times S^3$ and $S^7\times S^7$, where the kernel is $\mathbb{Z}/12\oplus\mathbb{Z}/12$ and $\mathbb{Z}/120\oplus\mathbb{Z}/120$ respectively. Related results for products are due to Booth and Heath \cite{booth-heath}, Heath \cite{heath} and Pave\v{s}i\'c \cite{pavesic}. For $S^m\times S^n$, Pave\v{s}i\'c \cite[Proposition~5.4]{pavesic} gives an exact sequence $0\to\operatorname{Coker}[\iota_m,-]\oplus\operatorname{Coker}[\iota_n,-]\to\mathcal{E}(S^m\times S^n)\to GL_2(\Lambda)$ whose left-hand group is the group of pinch-fold twists modulo Whitehead products that we use below, and Kahn \cite{kahn} treats self-equivalences of $(n-1)$-connected $2n$-manifolds. Sawashita \cite{sawashita} studies the group of self-equivalences of a product of two spheres. Crowley \cite{crowley} computes the extensions for the mapping class groups and for the homotopy equivalences of $\connsum_r(S^p\times S^p)$ with $p=3,7$. Baues \cite{baues} describes the group of homotopy equivalences of a manifold through a different set of exact sequences.

In dimension $4$, Hambleton and Kreck \cite{hk2} introduced a \emph{braid of interlocking exact sequences} that organizes the computation of $\Aut_\bullet(M)$ in terms of spin bordism groups, the Postnikov $2$-type $P_2(M)$, and the group $\Aut_\bullet(P_2(M))$. The framework has been applied to $4$-manifolds with $PD_2$ and free fundamental groups \cite{pamuk1,pamuk2} and to $s$-cobordism classifications \cite{hpr}. We extend the braid to higher dimensions. It does not require $M$ to be a product or an $H$-space.

\subsection*{The braid}

A closed oriented $m$-manifold $M$ is a \emph{$BO\langle n+1\rangle$-manifold} if its stable normal bundle $M\to BSO$ lifts to the $n$-connected cover $BO\langle n+1\rangle$ of $BSO$. For $n=2,3$ this means that $M$ is spin, and for $n=4,\dots,7$ that $M$ is string. Write $\B$ for $BO\langle n+1\rangle$, let $P=P_n(M)$ be the Postnikov $n$-type of $M$, let $c\colon M\to P$ be the canonical map, and let $X$ be the homotopy fibre of $c$. We assume that $\pi_1(M)$ is finite or infinite cyclic, and:
\begin{itemize}
\item[(T)] $\pi_1(P)$ acts trivially on $\pi_i(X)$ for all $i$.
\end{itemize}
Condition (T) holds when $M$ is simply connected, and when $c$ is a trivial fibration, as for $S^1\times S^{m-1}$ (Section~\ref{sec:braid}).

\begin{thma}[Theorem~\ref{thm:braid_exists}]
Let $M$ be a closed, connected, oriented $\B$-manifold of dimension $m=2n$ or $2n-1$, with $n\ge2$ and $m\ge4$, such that $\pi_1(M)$ is finite or infinite cyclic and (T) holds. Then the following diagram of exact sequences commutes, except that the two maps $\Omega^{\B}_{m+1}(P)\to\Aut_\bullet(M)$ agree only up to inversion:
\vspace{.5cm}
$$
\begin{matrix}
\xymatrix@!C@C-35pt{ \Om^{\B}_{m+1}(M) \ar[dr] \ar@/^2pc/[rr] &&
\htildeM \ar[dr] \ar@/^2pc/[rr] &&
\hept{P}\ar[dr]^{\beta} \\
&\Om^{\B}_{m+1}(P) \ar[dr] \ar[ur] &&
\hept{M}  \ar[dr]^{\alpha}\ar[ur] &&\Om^{\B}_m(P)\\
\pi_1(\heqpt{P}) \ar[ur] \ar@/_2pc/[rr] && \rOm^{\B}_{m+1}(P, M)
\ar[ur]^\gamma \ar@/_2pc/[rr]&& \rOm^{\B}_m(M)\ar[ur] }
\end{matrix}
$$
\vspace{1cm}

\noindent In particular there is an exact sequence of pointed sets
\[
  \pi_1(\heqpt{P})\to\widehat\Omega^{\B}_{m+1}(P,M)\xrightarrow{\ \gamma\ }\Aut_\bullet(M)\xrightarrow{\ \pi\ }\Aut_\bullet(P),
\]
and every element of the image of $\pi$ preserves the class $c_*[M]\in H_m(P;\mathbb{Z})$.
\end{thma}

The group $\widehat\Omega^{\B}_{m}(M)$ is the subgroup of $\B$-bordism classes over $M$ whose reference map has degree zero, and $\widehat\Omega^{\B}_{m+1}(P,M)$ is the corresponding relative group (Section~\ref{sec:braid}). The set $\widetilde{\mathcal{H}}(M)$ consists of classes of $\B$-cobordisms from $M$ to $M$ that restrict to a self-equivalence at one end. The map $\gamma$ is defined by obstruction-theoretic lifting, and $\alpha$ and $\beta$ compare a self-equivalence with the identity in bordism. Thus $\ker\pi$ is the image of a bordism group, and the Atiyah--Hirzebruch spectral sequence bounds that group.

The proofs follow Hambleton and Kreck. The surgeries needed to make the maps to $P$ highly connected lie below the middle dimension, and the obstruction groups vanish for $m\in\{2n-1,2n\}$. The normal $\B$-structure need only exist; it need not factor through $P$.

\subsection*{Applications to spheres}

Section~\ref{sec:applications} applies the braid to products and connected sums of spheres. Several of these groups can also be computed by direct homotopy-theoretic arguments, and we use both methods. The kernels for products of two spheres are instances of the results of Pave\v{s}i\'c mentioned above. For $S^3\times S^3$ the braid gives an independent bound that matches the elementary computation.

\begin{thmb}[Theorem~\ref{splitting}]
For $M=S^1\times S^{m-1}$ with $m\ge4$ we have $\Aut_\bullet(M)\cong\mathbb{Z}/2\oplus\mathbb{Z}/2$. Including orientation-reversing equivalences, the group of all based self-equivalences is $(\mathbb{Z}/2)^3$.
\end{thmb}

\begin{thmc}[Theorem~\ref{thm:S3xS3_split}]
For $M=S^3\times S^3$ there is a split short exact sequence
\[
  0\to(\mathbb{Z}/12)^2\xrightarrow{\ \gamma\ }\Aut_\bullet(S^3\times S^3)\xrightarrow{\ \pi\ }SL_2(\mathbb{Z})\to0,
\]
in which $SL_2(\mathbb{Z})$ acts on $(\mathbb{Z}/12)^2$ through its standard action modulo $12$. Hence $\Aut_\bullet(S^3\times S^3)\cong(\mathbb{Z}/12)^2\rtimes SL_2(\mathbb{Z})$. This group is not a direct product, and its abelianization is $\mathbb{Z}/12$. The relative spin bordism group $\widehat\Omega^{spin}_7(P,M)$ has order $144$ and $\gamma$ is a bijection onto the kernel.
\end{thmc}

The kernel agrees with the one in Sieradski's exact sequence \cite[Example~3]{sieradski}. The action of $SL_2(\mathbb{Z})$ on it and the splitting are obtained by identifying $\Aut_\bullet(S^3\times S^3)$ with a group of automorphisms of $[S^3\times S^3,S^3]_*$.

\begin{thmd}[Theorems~\ref{thm:S3S4} and~\ref{thm:S4S4}]
$\Aut_\bullet(S^3\times S^4)\cong(\mathbb{Z}/2\oplus\mathbb{Z}/24)\rtimes(\mathbb{Z}/2)^2$ is a group of order $192$, and $\Aut_\bullet(S^4\times S^4)\cong D_4\times\mathbb{Z}/2$ is a group of order $16$, where $D_4$ is the dihedral group of order $8$. In both cases the kernel of $\pi$ is the group of pinch-fold twists, which is computed from the Whitehead products in the homotopy groups of the factors.
\end{thmd}

\begin{thmf}[Theorem~\ref{thm:SpSp}]
For $M=S^p\times S^p$ with $p\ge2$, the kernel of $\pi$ is $Q_p\oplus Q_p$ with $Q_p=\pi_{2p}(S^p)/[\pi_{p+1}(S^p),\iota_p]$, and the image of $\pi$ is the group of isometries $A$ of the intersection form with $ab[\iota_p,\iota_p]=cd[\iota_p,\iota_p]=0$. In particular $\Aut_\bullet(S^5\times S^5)$ is the theta group $\Gamma_\theta\subset SL_2(\mathbb{Z})$, and $\Aut_\bullet(S^7\times S^7)$ is an extension of $SL_2(\mathbb{Z})$ by $(\mathbb{Z}/120)^2$.
\end{thmf}

\begin{thme}[Theorem~\ref{thm:Mr_kernel}]
For $M_r=\connsum_r(S^3\times S^3)$ with $r\ge1$, the image of $\pi\colon\Aut_\bullet(M_r)\to\Aut_\bullet(P)$ is $Sp_{2r}(\mathbb{Z})$ and the kernel is $(\mathbb{Z}/12)^{2r}$, generated by twists supported on single spheres. For $r\ge2$ the diffeomorphisms that permute the summands act on the kernel by permuting the $r$ pairs of generators, so the conjugation action of $\Aut_\bullet(M_r)$ on the kernel is not trivial. By Crowley \cite{crowley} the resulting extension of $Sp_{2r}(\mathbb{Z})$ by the kernel does not split for $r\ge2$.
\end{thme}

\medskip
All manifolds are smooth, closed, connected and oriented unless stated otherwise, and maps between pointed spaces preserve basepoints. The group $\Aut_\bullet(M)$ consists of orientation-preserving equivalences unless stated otherwise. For based spaces $X$ and $Y$, $[X,Y]_\bullet$ denotes the set of based homotopy classes of based maps $X\to Y$. Abelian groups are written additively.

\section{Construction of the braid}
\label{sec:braid}

We follow Hambleton and Kreck \cite[Section~2]{hk2}, who treat the case $m=4$, $n=2$. The dimension bookkeeping in the surgery and obstruction-theory steps depends on $m$ and $n$, so we give the arguments in full.

\subsection{Conventions}
\label{subsec:conventions}

Throughout, $n\ge2$ and $m\in\{2n-1,2n\}$ with $m\ge4$. Put $\B=BO\langle n+1\rangle$, the $n$-connected cover of $BSO$; thus $\pi_i(\B)=0$ for $i\le n$ and $\B\to BSO$ induces isomorphisms on $\pi_i$ for $i>n$. A \emph{$\B$-manifold} is a smooth closed (or compact, with boundary) oriented manifold together with a lift of its stable normal bundle $N\to BSO$ to $\B$, called a $\B$-structure. For a space $Y$ the group $\Omega^{\B}_k(Y)$ is the bordism group of triples $(N,\nu_N,g)$ with $N$ a closed $k$-dimensional $\B$-manifold and $g\colon N\to Y$ a map; it is the value on $Y$ of the generalized homology theory $\pi_k(M\B\wedge Y_+)$, where $M\B$ is the Thom spectrum. Disjoint union is the sum, and $-[N,g]=[-N,g]$. Reference maps are based when $Y$ is based. Abbreviate $\Omega_k:=\Omega^{\B}_k(*)$.

Let $M$ be a closed, connected, oriented $\B$-manifold of dimension $m$, with a fixed $\B$-structure $\nu$, whose fundamental group is finite or infinite cyclic. Let $P=P_n(M)$ be the Postnikov $n$-type of $M$ and $c\colon M\to P$ the canonical map. It induces isomorphisms on $\pi_i$ for $i\le n$, and $P$ is obtained from $M$ by attaching cells of dimension at least $n+2$. Hence $c$ is $(n+1)$-connected, and the homotopy fibre $X$ of $c$ is $n$-connected. We replace $P$ by the mapping cylinder of $c$, so that $M\subset P$ and $c$ is the inclusion. For lifting problems we use that $c$ is homotopic to a fibration with fibre $X$. Let $V$ be a space and $A\subset V$ a subcomplex. Given a map $\Phi\colon V\to P$ and a map $\Phi_A\colon A\to M$ with $c\circ\Phi_A\simeq\Phi|_A$, a \emph{lift of $\Phi$ extending $\Phi_A$} is a map $r\colon V\to M$ with $r|_A=\Phi_A$ together with a homotopy $c\circ r\simeq\Phi$ that is stationary on $A$. By the homotopy extension property we may replace $\Phi$ by a homotopic map with $c\circ\Phi_A=\Phi|_A$ strictly; this does not change any bordism class below, since reference maps matter only up to homotopy. Two lifts are \emph{homotopic rel $A$} if they are homotopic by a homotopy stationary on $A$.

The only further hypothesis is the following.
\begin{itemize}
\item[(T)] $\pi_1(P)$ acts trivially on $\pi_i(X)$ for every $i$.
\end{itemize}
It holds if $M$ is simply connected, and also if $c$ is a trivial fibration, as for $M=S^1\times S^{m-1}$, where $c$ is the projection to the first factor. Condition (T) says that the local coefficient systems $\pi_i(X)$ over $P$ are constant, so the obstruction groups below are ordinary cohomology groups.

\begin{remark}\label{rem:no-F}
The $\B$-structure $\nu$ need not factor through $c$. All bordism groups below are taken with the fixed structure $\nu$ on $M$, and the reference maps to $P$ do not depend on $\B$-structures. A factorization would be needed only to identify $\Omega^{\B}_*(P)$ with the bordism group of the normal $n$-type in Kreck's sense.
\end{remark}

\begin{remark}
For $m=4$ and $n=2$ we have $\B=BSpin$, and Theorem~A recovers the braid of Hambleton and Kreck for spin $4$-manifolds, with the additional restriction on $\pi_1$ made here for the proof of Lemma~\ref{lem:PD_upgrade}. For $n=3$ the manifolds are spin and for $4\le n\le7$ they are string, since $BO\langle5\rangle=BO\langle6\rangle=BO\langle7\rangle=BO\langle8\rangle=BString$.
\end{remark}

In every construction below the cobordism $W$ contains an arc joining the basepoints of its two ends, along which all maps to $M$ and $P$ are based. Surgeries are performed away from this arc. If the other end of the cobordism is disconnected, the arc ends in a chosen component and the lifts are based there, which is possible because the fibre $X$ is connected. The arc is contractible, so it does not change homology in positive degrees, and the obstruction groups, which lie in degrees at least $n+2\ge4$, are unaffected.

\subsection{Surgery and lifting}
\label{subsec:surgery-lifting}

For $d\in\{m+1,m+2\}$ put $\ell_d=\lfloor d/2\rfloor-1$, the largest $i$ with $2i+2\le d$. Then $\ell_{m+1}=n-1$ and $\ell_{m+2}=m-n\le n$. The condition $2i+2\le d$ ensures that surgery on an $i$-sphere in a $d$-manifold creates no new homotopy in degrees at most $i$.

Let $V$ be a compact connected oriented $\B$-manifold of dimension $d$ whose boundary is written $\partial V=A\cup C$, with $A$ and $C$ compact submanifolds of codimension $0$ in $\partial V$ and $A\cap C=\partial A=\partial C$. Let $\Phi\colon V\to P$ be a map such that $C$ is connected and $\Phi|_C\colon C\to P$ induces isomorphisms on $\pi_j$ for $j\le n$ (for example $C=M$ with $\Phi|_C=c$, or $C=M\times I$ with $\Phi|_C=c\circ\mathrm{pr}_1$).

\begin{lemma}[Surgery]\label{lem:surgery}
After a sequence of surgeries on embedded spheres in the interior of $V$, the triple $(V,\Phi)$ is replaced by a $\B$-bordant triple $(V',\Phi')$ rel $\partial V$ with $V'$ connected, $\Phi'|_{\partial V}=\Phi|_{\partial V}$, and such that $\Phi'_*\colon\pi_j(V')\to\pi_j(P)$ is an isomorphism for $j\le\ell_d$. In that case the pair $(V',C)$ is $\ell_d$-connected, so $H_j(V',C;G)=0$ for $j\le\ell_d$ and every abelian group $G$ with trivial action.
\end{lemma}
\begin{proof}
Since $\Phi|_C$ is an isomorphism on $\pi_j$ for $j\le n$ and $\ell_d\le n$, the map $\Phi_*$ is onto on $\pi_j$ for $j\le\ell_d$. We kill $\ker(\Phi_*\colon\pi_j(V)\to\pi_j(P))$ for $j=1,2,\dots,\ell_d$ in turn. For $j=1$ the kernel is normally generated by finitely many elements, because $\pi_1(V)$ is finitely presented and $\pi_1(P)=\pi_1(M)$ is finitely presented. For $j\ge2$ the previous steps make the universal cover of $V$ $(j-1)$-connected, so $\pi_j(V)$ is a finitely generated module over the Noetherian ring $\mathbb{Z}\pi_1(M)$ (here we use that $\pi_1(M)$ is finite or infinite cyclic). Indeed, $P$ is obtained from $M$ by attaching cells of dimension at least $n+2$, so the lift $\widetilde\Phi\colon\widetilde V\to\widetilde P$ to universal covers has $\pi_{j+1}(\widetilde P,\widetilde V)\cong H_{j+1}(\widetilde P,\widetilde V)$ because the pair is $j$-connected, and this is finitely generated because $H_{j+1}(\widetilde P,\widetilde M)=0$ for $j+1\le n+1$ while $V$ and $M$ are finite complexes up to homotopy. The exact sequence $\pi_{j+1}(\widetilde P,\widetilde V)\to\pi_j(V)\to\pi_j(P)$ then shows that the kernel of $\Phi_*$ on $\pi_j$ is a quotient of a finitely generated module, hence is finitely generated. Represent each generator by an embedded sphere $S^j\subset\mathrm{int}\,V$; this is possible because $2j<d$. Since $2j+2\le d$, the dual handle of the surgery has index $d-j\ge j+2$, so surgery on $S^j$ does not create elements of $\pi_i$ for $i\le j$. Since $\pi_j(\B)=0$ for $j\le n$, the $\B$-structure restricted to $S^j$ is null-homotopic, and the choice of a null-homotopy gives a stable framing of the normal bundle of $S^j$ in $V$. In the stable range $d-j>j$ it is an honest framing, for which the $\B$-structure extends over the trace of the surgery. Since $\Phi$ is null-homotopic on $S^j$, it extends over the trace as well. Surgery on the sphere replaces $\pi_j(V)$ by its quotient by the submodule generated by the class of the sphere and does not change $\pi_i$ for $i<j$. Since the sphere lies in the kernel of $\Phi_*$, the map $\Phi$ induces on the quotient the map to $\pi_j(P)$ induced by $\Phi_*$, which is still onto. As the spheres kill the whole kernel, it is injective as well. After the last step $\Phi'_*$ is an isomorphism on $\pi_j$ for $j\le\ell_d$. The surgered manifold is connected because we only perform surgeries on spheres of dimension at least $1$ in a manifold of dimension at least $5$.

Finally $\pi_j(C)\to\pi_j(V')$ is an isomorphism for $j\le\ell_d$, because both $\pi_j(C)$ and $\pi_j(V')$ map isomorphically to $\pi_j(P)$ and the composite is the map induced by $\Phi|_C$. Hence $\pi_j(V',C)=0$ for $j\le\ell_d$, and by the relative Hurewicz theorem and the universal coefficient theorem $H_j(V',C;G)=0$ for $j\le\ell_d$.
\end{proof}

\begin{lemma}[Lifting]\label{lem:lifting}
Assume (T). Let $(V,\Phi)$ be as in Lemma~\ref{lem:surgery} with $d\in\{m+1,m+2\}$, with $\Phi_*$ an isomorphism on $\pi_j$ for $j\le\ell_d$, and let $\Phi_A\colon A\to M$ be a map with $c\circ\Phi_A\simeq\Phi|_A$. Then there is a lift of $\Phi$ extending $\Phi_A$.
\end{lemma}
\begin{proof}
Lifts in the sense of the definition are what obstruction theory for the fibration replacing $c$ produces. Since $X$ is $n$-connected, the obstructions to lifting rel $A$ lie in $H^{i+1}(V,A;\pi_i(X))$ for $i\ge n+1$. The coefficient systems are pulled back from $P$ and are constant by (T), and the fibre $X$ is simply connected, so ordinary obstruction theory applies. By Poincar\'e--Lefschetz duality $H^{i+1}(V,A;G)\cong H_{d-i-1}(V,C;G)$. By Lemma~\ref{lem:surgery} these groups vanish when $d-i-1\le\ell_d$. The largest degree occurring is $d-n-2$, which is $m-n-1\le\ell_{m+1}$ for $d=m+1$ and $m-n\le\ell_{m+2}$ for $d=m+2$.
\end{proof}

\begin{lemma}\label{lem:PD_upgrade}
Let $M$ be a closed oriented manifold whose fundamental group $\pi$ is finite or infinite cyclic, and let $f\colon M\to M$ be a map of degree one that induces an isomorphism on $\pi_1$. Then $f$ is a homotopy equivalence.
\end{lemma}
\begin{proof}
Let $\widetilde f\colon\widetilde M\to\widetilde M$ be a lift to the universal cover; it is equivariant with respect to the automorphism $f_*$ of $\pi$. Since $f$ has degree one and induces an isomorphism on $\pi_1$, the usual cap product argument with $\mathbb{Z}\pi$-coefficients shows that $\widetilde f_*\colon H_j(\widetilde M)\to H_j(\widetilde M)$ is a split surjection for every $j$ (see \cite[Chapter~2]{wall_surgery}). Explicitly, Poincar\'e duality for the closed manifold $M$ with $\mathbb{Z}\pi$-coefficients gives isomorphisms $H_j(\widetilde M)\cong H^{m-j}(M;\mathbb{Z}\pi)$, and the Umkehr map $f^!\colon H_j(\widetilde M)\to H_j(\widetilde M)$ obtained from $f^*$ on $H^{m-j}(M;\mathbb{Z}\pi)$ satisfies $\widetilde f_*\circ f^!=\id$, because $f_*[M]=[M]$ and $f_*$ is an isomorphism on $\pi_1$. The automorphism group of $\pi$ is finite, so some power $f^k$ induces the identity on $\pi$, and $(\widetilde f_*)^k=(\widetilde f^{\,k})_*$ is a surjective $\mathbb{Z}\pi$-linear endomorphism of $H_j(\widetilde M)$. The ring $\mathbb{Z}\pi$ is Noetherian and $\widetilde M$ has a finite free $\mathbb{Z}\pi$-chain complex, so $H_j(\widetilde M)$ is a Noetherian module, and a surjective endomorphism of a Noetherian module is injective. Hence $\widetilde f_*$ is injective, and therefore bijective. Thus $\widetilde f$ is a homology isomorphism of simply connected spaces, hence a homotopy equivalence, and so is $f$.
\end{proof}

\subsection{The maps $\alpha$ and $\beta$, and relative bordism}
\label{subsec:bordism-alpha}

Write $[N,g]$ for the class of $(N,\nu_N,g)$, and equip $M$ with its fixed $\B$-structure $\nu$. The \emph{degree} of a class $[N,g]\in\Omega^{\B}_m(M)$ is $g_*[N]\in H_m(M;\mathbb{Z})\cong\mathbb{Z}$; it is a homomorphism. Let
\[
\widehat\Omega^{\B}_m(M):=\ker\bigl(\deg\colon\Omega^{\B}_m(M)\to\mathbb{Z}\bigr).
\]
Define
\[
\alpha\colon\Aut_\bullet(M)\to\widehat\Omega^{\B}_m(M),\qquad\alpha(f)=[M,f]-[M,\id],
\]
and
\[
\beta\colon\Aut_\bullet(P)\to\Omega^{\B}_m(P),\qquad\beta(\phi)=[M,\phi\circ c]-[M,c].
\]
Here $\Aut_\bullet(P)$ is the group of based homotopy classes of self-equivalences of $P$. The map $\alpha$ lands in $\widehat\Omega^{\B}_m(M)$ because $f_*[M]=[M]$. Neither map is a homomorphism: for $f,g\in\Aut_\bullet(M)$ one has $\alpha(f\circ g)=\alpha(f)+f_*\alpha(g)$.

Every self-equivalence $f$ of $M$ induces a unique $\phi_f\in\Aut_\bullet(P)$ with $\phi_f\circ c\simeq c\circ f$, because $P$ is an $n$-type and $c$ is $(n+1)$-connected, so that $c^*\colon[P,P]_\bullet\to[M,P]_\bullet$ is a bijection. This defines a homomorphism
\[
\pi\colon\Aut_\bullet(M)\to\Aut_\bullet(P),\qquad\pi(f)=\phi_f .
\]
Since $\phi_f\circ c\simeq c\circ f$ and $f_*[M]=[M]$, every element of the image of $\pi$ preserves $c_*[M]\in H_m(P;\mathbb{Z})$.

Regard $M\subset P$ via the mapping cylinder. The relative bordism group $\Omega^{\B}_{m+1}(P,M)$ consists of classes of triples $(W,\nu_W,F)$, where $W$ is a compact $(m+1)$-dimensional $\B$-manifold with boundary, $F\colon(W,\partial W)\to(P,M)$ is a map, and bordisms are bordisms of such pairs. There is a boundary map $\partial\colon\Omega^{\B}_{m+1}(P,M)\to\Omega^{\B}_m(M)$, $\partial[W,F]=[\partial W,F|_{\partial W}]$, and the long exact sequence of the pair is exact. Define
\[
\widehat\Omega^{\B}_{m+1}(P,M):=\partial^{-1}\bigl(\widehat\Omega^{\B}_m(M)\bigr),
\]
the group of relative classes whose boundary has degree zero over $M$. The next statement is then immediate from the long exact sequence of the pair.

\begin{lemma}\label{bordism exact}
The sequence
\[
\Omega^{\B}_{m+1}(M)\to\Omega^{\B}_{m+1}(P)\to\widehat\Omega^{\B}_{m+1}(P,M)\xrightarrow{\ \partial\ }\widehat\Omega^{\B}_m(M)\to\Omega^{\B}_m(P)
\]
is exact.
\end{lemma}
\begin{proof}
The sequence is the long exact sequence of the pair $(P,M)$ in the homology theory $\Omega^{\B}_*$, with $\Omega^{\B}_m(M)$ and $\Omega^{\B}_{m+1}(P,M)$ replaced by the subgroups $\widehat\Omega^{\B}_m(M)$ and $\partial^{-1}\widehat\Omega^{\B}_m(M)$. Exactness at $\Omega^{\B}_{m+1}(P)$ is unchanged. At $\widehat\Omega^{\B}_{m+1}(P,M)$ the kernel of $\partial$ is the image of $\Omega^{\B}_{m+1}(P)$, which lies in the subgroup. At $\widehat\Omega^{\B}_m(M)$ the kernel of the map to $\Omega^{\B}_m(P)$ is $\widehat\Omega^{\B}_m(M)\cap\operatorname{im}\partial$, which is $\partial$ of $\partial^{-1}\widehat\Omega^{\B}_m(M)$ by definition.
\end{proof}

\subsubsection*{Normal form for relative classes}

Say that $(W,F)$ is in \emph{standard form} if $W$ is connected, $\partial W=\partial_1W\sqcup N$ with $\partial_1W=-M$ and $F|_{\partial_1W}=c$, and $g:=F|_N\colon N\to M$ has degree $1$, that is $g_*[N]=[M]$. Here $N$ is any closed $\B$-manifold, possibly disconnected.

\begin{lemma}\label{lem:normal_form}
Every class in $\widehat\Omega^{\B}_{m+1}(P,M)$ has a representative in standard form. If two standard form representatives $(W,F)$ and $(W',F')$ with boundaries $-M\sqcup N$ and $-M\sqcup N'$ define the same class, then there is a $\B$-manifold $T$ of dimension $m+2$ with
\[
\partial T=W\cup W'\cup(M\times I)\cup Y_N ,
\]
where $M\times I$ joins $\partial_1W$ to $\partial_1W'$, $Y_N$ is a $\B$-bordism between $N$ and $N'$, and $T$ carries a map to $P$ extending $F$, $F'$, $c\circ\mathrm{pr}_1$ on $M\times I$, and a map $G_N\colon Y_N\to M$.
\end{lemma}
\begin{proof}
Let $[W_0,F_0]\in\widehat\Omega^{\B}_{m+1}(P,M)$, so $[\partial W_0,F_0|_{\partial W_0}]$ has degree $0$ over $M$. Form $W_1=W_0\sqcup(M\times I)$ with the map $F_0\sqcup c\circ\mathrm{pr}_1$. The class of $(M\times I,c\circ\mathrm{pr}_1)$ in $\Omega^{\B}_{m+1}(P,M)$ is zero: it is the part $M\times I\times\{0\}$ of the boundary of $M\times I\times I$, and the rest of that boundary maps to $M$ under $c\circ\mathrm{pr}_1$. So $W_1$ represents the same class. Put $\partial_1W_1=-M\times\{0\}$ and $N_1=\partial W_0\sqcup M\times\{1\}$, with the map $F_0|_{\partial W_0}\sqcup\mathrm{id}$; it has degree $0+1=1$. Finally join all components of $W_1$ by boundary connected sums along arcs ending in $N_1$ (and by interior connected sums with closed components). This does not change the relative bordism class, and it keeps $\partial_1W_1$ and the degree. The result is a standard form representative.

For the second assertion, a relative bordism between the two representatives is a $\B$-manifold $T$ with $\partial T=W\cup Y\cup W'$, where $Y$ is a bordism over $M$ between $\partial W=-M\sqcup N$ and $\partial W'=-M\sqcup N'$, and the map to $P$ extends the given maps on $W$, $W'$ and maps $Y$ into $M$. In $\Omega^{\B}_m(M)$ we have $[-M,\id]+[N,g]=[-M,\id]+[N',g']$, and cancelling in this group gives $[N,g]=[N',g']$, so there is a bordism $Y_N$ over $M$ between $N$ and $N'$. Let $Y'=(M\times I)\sqcup Y_N$. The closed manifold $\Sigma=Y\cup(-Y')$ over $M$ satisfies $[W\cup Y'\cup W']=[W\cup Y\cup W']-[\Sigma]=-[\Sigma]$ in $\Omega^{\B}_{m+1}(P)$, because $T$ is a null-bordism of $W\cup Y\cup W'$. Replacing $Y_N$ by $Y_N\sqcup\Sigma$, which is again a bordism over $M$ between $N$ and $N'$, makes $[W\cup Y'\cup W']=0$, and a null-bordism of this closed manifold is the required $T$.
\end{proof}

\subsection{The groups $\widetilde{\mathcal H}(M)$ and $\widetilde{\mathcal H}(P)$}
\label{subsec:He}

\begin{definition}
$\widetilde{\mathcal H}(M)$ is the set of equivalence classes of pairs $(W,F)$, where $W$ is a compact $(m+1)$-dimensional $\B$-manifold with $\partial W=\partial_1W\sqcup\partial_2W$, $\partial_1W=-M$, $\partial_2W=M$ (with the given $\B$-structure on both ends), and $F\colon W\to M$ is a map with $F|_{\partial_1W}=\id_M$ and $F|_{\partial_2W}=f$ a based homotopy equivalence. Two pairs $(W,F)$, $(W',F')$ are equivalent if there is a based homotopy $h$ from $f=F|_{\partial_2W}$ to $f'=F'|_{\partial_2W'}$ such that the closed manifold
\[
-W'\cup_{\partial_1}W\cup_{\partial_2}(M\times I)
\]
with the map $F'\cup F\cup h$ is zero in $\Omega^{\B}_{m+1}(M)$. The base point is the class of $(M\times I,\mathrm{pr}_1)$.

$\widetilde{\mathcal H}(P)$ is defined in the same way with $F\colon W\to P$, $F|_{\partial_1W}=c$ and $F|_{\partial_2W}=f$ a based map $M\to P$ that is an $(n+1)$-equivalence, that is, $f=\phi\circ c$ for a unique $\phi\in\Aut_\bullet(P)$ (so $f\simeq\phi\circ c$ is an iso on $\pi_i$, $i\le n$); the homotopy $h$ is a homotopy in $P$, and the closed manifold has to vanish in $\Omega^{\B}_{m+1}(P)$.
\end{definition}

Both sets are groups under composition of cobordisms, as in \cite{hk2}. We shall only use the pointed set structure. The next lemma is the analogue of \cite[Lemma~2.5]{hk2}.

\begin{lemma}\label{exactA}
There is an exact sequence of pointed sets
\[
\Omega^{\B}_{m+1}(M)\to\widetilde{\mathcal H}(M)\to\Aut_\bullet(M)\xrightarrow{\ \alpha\ }\widehat\Omega^{\B}_m(M).
\]
\end{lemma}
\begin{proof}
The first map sends $[N,g]$ to the class of $(M\times I\sqcup N,\ \mathrm{pr}_1\sqcup g)$, and the second sends $(W,F)$ to $[F|_{\partial_2W}]$. It is well defined, because equivalent pairs have homotopic restrictions. The composite of the first two maps is trivial, since the restriction is the identity.

If $(W,F)\mapsto[\id_M]$, then $f\simeq\id_M$ by a based homotopy $h$. Then $Z:=W\cup(M\times I)$, with the ends of $M\times I$ glued to $\partial_2W$ and $\partial_1W$, and the map $F\cup h$, is a closed $\B$-manifold over $M$, and $(W,F)$ is equivalent to the image of $[Z,F\cup h]$, because $-(M\times I\sqcup Z)\cup W\cup(M\times I)$ is bordant over $M$ to $Z-Z=0$. This is exactness at $\widetilde{\mathcal H}(M)$.

Finally $\alpha(f)=0$ means that $(M,\nu,f)$ and $(M,\nu,\id)$ are bordant over $M$. A bordism is a pair $(W,F)$ with $\partial W=-M\sqcup M$, $F|_{\partial_1W}=\id$ and $F|_{\partial_2W}=f$, that is an element of $\widetilde{\mathcal H}(M)$ mapping to $[f]$. Conversely such an element gives a bordism. This is exactness at $\Aut_\bullet(M)$.
\end{proof}

\begin{lemma}\label{HeM=HeP}
Assume (T). The map $c_\sharp\colon\widetilde{\mathcal H}(M)\to\widetilde{\mathcal H}(P)$, $(W,F)\mapsto(W,c\circ F)$, is a bijection.
\end{lemma}
\begin{proof}
\emph{Well defined.} If $f=F|_{\partial_2W}$ is a self-equivalence of $M$, then $c\circ f=\phi_f\circ c$ is an $(n+1)$-equivalence. A homotopy $h$ in $M$ gives the homotopy $c\circ h$, and $c$ maps the null-bordism in $\Omega^{\B}_{m+1}(M)$ to one in $\Omega^{\B}_{m+1}(P)$.

\emph{Surjective.} Let $(W,F)\in\widetilde{\mathcal H}(P)$ with $F|_{\partial_2W}=\phi\circ c$. We may assume $W$ connected (join components by interior connected sums, which does not change the class). Apply Lemma~\ref{lem:surgery} with $A=\partial_1W$, $C=\partial_2W$ and $d=m+1$; note that $\Phi|_C=\phi\circ c$ is an isomorphism on $\pi_j$ for $j\le n$. After surgery in the interior, which changes $(W,F)$ within its class in $\widetilde{\mathcal H}(P)$ (the trace of a surgery rel boundary is a $\B$-cobordism rel boundary, so the closed manifold in the equivalence relation is null-bordant), we may assume that $F_*$ is an isomorphism on $\pi_j$ for $j\le\ell_{m+1}$. Lemma~\ref{lem:lifting}, applied to the lift $\id_M$ on $A=\partial_1W$, gives a lift $\widehat F\colon W\to M$ of $F$ with $\widehat F|_{\partial_1W}=\id_M$. Put $\widehat f=\widehat F|_{\partial_2W}$. Then $(W,\widehat F)$ is a $\B$-bordism over $M$ between $(M,\id)$ and $(M,\widehat f)$, so $\widehat f_*[M]=[M]$, so $\widehat f$ has degree one. Also $c\circ\widehat f\simeq\phi\circ c$ is an isomorphism on $\pi_j$ for $j\le n$, so $\widehat f$ is an isomorphism on $\pi_1$, because $n\ge2$. By Lemma~\ref{lem:PD_upgrade}, $\widehat f$ is a homotopy equivalence, so $(W,\widehat F)\in\widetilde{\mathcal H}(M)$ and $c_\sharp(W,\widehat F)=(W,c\widehat F)$ is equivalent to $(W,F)$ by the homotopy $c\widehat F\simeq F$.

\emph{Injective.} Let $(W,F),(W',F')\in\widetilde{\mathcal H}(M)$ with $c_\sharp$-images equivalent. Then there is a based homotopy $h_P\colon cf\simeq cf'$ in $P$ such that $Z_P:=-W'\cup W\cup(M\times I)$ with the map $cF'\cup cF\cup h_P$ bounds a $\B$-manifold $(V,\Phi)$ over $P$. Replacing disjoint unions by connected sums gives equivalent pairs, so we may assume $W$ and $W'$ are connected; then $Z_P$ is connected, and we may take $V$ connected. Put $A:=W\cup W'$ (the union of the two copies glued along $\partial_1=M$, so $A$ is connected) and $C:=M\times I$, so $\partial V=A\cup C$ with $A\cap C=M\times\{0,1\}$. The map $\Phi|_A$ has the lift $F\cup F'$ to $M$. The map $\Phi|_C$ is the homotopy $h_P$, which is an isomorphism on $\pi_j$ for $j\le n$ because its ends are. By Lemma~\ref{lem:surgery} we may assume $\Phi_*$ is an isomorphism for $j\le\ell_{m+2}$ ($\dim V=m+2$). By Lemma~\ref{lem:lifting} there is a lift $\widehat\Phi\colon V\to M$ extending $F\cup F'$ on $A$. Its restriction $h:=\widehat\Phi|_C\colon M\times I\to M$ is a based homotopy from $f$ to $f'$, and $(V,\widehat\Phi)$ is a null-bordism over $M$ of the closed manifold $-W'\cup W\cup(M\times I)$ with the map $F'\cup F\cup h$. So $(W,F)\sim(W',F')$ in $\widetilde{\mathcal H}(M)$.
\end{proof}

\begin{remark}
Both the obstruction groups in Lemma~\ref{lem:lifting} and the connectivity range $\ell_d$ of the surgery lemma depend on (T) and on the restriction $m\in\{2n-1,2n\}$. For $m=2n+1$ and $d=m+2$ the existence group $H_{n+1}(V,C;\pi_{n+1}(X))$ is a middle-dimensional obstruction group that surgery cannot kill, and for $m\ge2n+2$ the existence groups $H_{d-i-1}(V,C;\pi_i(X))$ with $i=n+1$ lie above $\ell_d$.
\end{remark}

\begin{lemma}\label{exactB}
There is an exact sequence of pointed sets
\[
\pi_1(\mathcal E_\bullet(P))\to\Omega^{\B}_{m+1}(P)\to\widetilde{\mathcal H}(M)\to\Aut_\bullet(P)\xrightarrow{\ \beta\ }\Omega^{\B}_m(P).
\]
\end{lemma}
\begin{proof}
We use $\widetilde{\mathcal H}(M)\cong\widetilde{\mathcal H}(P)$ (Lemma~\ref{HeM=HeP}). The map $\widetilde{\mathcal H}(P)\to\Aut_\bullet(P)$ sends $(W,F)$ to $\phi$, where $F|_{\partial_2W}=\phi\circ c$. It is well defined, since a homotopy $f\simeq f'$ gives $\phi=\phi'$. The map $\Omega^{\B}_{m+1}(P)\to\widetilde{\mathcal H}(P)$ sends $[N,g]$ to $(M\times I\sqcup N,\ c\circ\mathrm{pr}_1\sqcup g)$. A loop $\lambda$ in $\mathcal E_\bullet(P)=\mathrm{Map}_\bullet(P,P)_{\simeq}$ based at $\id$ gives $\widehat\lambda\colon P\times S^1\to P$, and $\pi_1(\mathcal E_\bullet(P))\to\Omega^{\B}_{m+1}(P)$ sends $\lambda$ to $[M\times S^1,\ \widehat\lambda\circ(c\times\id)]$, where $M\times S^1$ carries the product of $\nu$ with the bounding $\B$-structure on $S^1$.

\emph{Exactness at $\Aut_\bullet(P)$.} Suppose $(W,F)\mapsto\phi$. Then $W$ is a bordism over $P$ from $(M,c)$ to $(M,\phi c)$, so $\beta(\phi)=0$. Conversely, if $\beta(\phi)=0$, a bordism over $P$ from $(M,c)$ to $(M,\phi c)$ is a pair in $\widetilde{\mathcal H}(P)$ with image $\phi$.

\emph{Exactness at $\widetilde{\mathcal H}(P)$.} The composite $\Omega^{\B}_{m+1}(P)\to\widetilde{\mathcal H}(P)\to\Aut_\bullet(P)$ is trivial. If $(W,F)\mapsto\id_P$, then $f=F|_{\partial_2W}\simeq c$ by a based homotopy $h$, and $Z=W\cup(M\times I)$, with the ends of $M\times I$ glued to $\partial_2W$ and $\partial_1W$, and the map $F\cup h$, is a closed manifold over $P$. As in the proof of Lemma~\ref{exactA}, $(W,F)$ is equivalent to the image of $[Z,F\cup h]$.

\emph{Exactness at $\Omega^{\B}_{m+1}(P)$.} A class $[N,g]$ maps to the base point iff there is a homotopy $h$ from $c$ to $c$, that is, a loop $\lambda$ in $\mathrm{Map}_\bullet(M,P)$ based at $c$, such that $(M\times I\sqcup N)\cup(M\times I)$ with the map $c\circ\mathrm{pr}_1\sqcup g\cup h$ vanishes in $\Omega^{\B}_{m+1}(P)$. This closed manifold is $N\sqcup M\times S^1$ with the maps $g$ and the loop, so the condition says $[N,g]=-[M\times S^1,h]$. Since $c^*\colon\mathrm{Map}_\bullet(P,P)\to\mathrm{Map}_\bullet(M,P)$ is a weak equivalence ($P$ is an $n$-type and $c$ is $(n+1)$-connected), every loop $\lambda$ in $\mathrm{Map}_\bullet(M,P)$ at $c$ is of the form $c^*$ of a loop in $\mathcal E_\bullet(P)$ at $\id$, and $-[M\times S^1,h]=[M\times S^1,\bar h]$ for the reversed loop. Hence the kernel is the image of $\pi_1(\mathcal E_\bullet(P))$.
\end{proof}

Write $\theta\colon\Omega^{\B}_{m+1}(P)\to\widetilde{\mathcal H}(M)$ for the composite of the map of Lemma~\ref{exactB} with the bijection of Lemma~\ref{HeM=HeP}, and $\iota\colon\Omega^{\B}_{m+1}(M)\to\widetilde{\mathcal H}(M)$ for the first map of Lemma~\ref{exactA}. Then $\theta\circ c_*=\iota$ by definition.

\begin{lemma}\label{lem:fibres}
For $x,x'\in\Omega^{\B}_{m+1}(P)$ we have $\theta(x)=\theta(x')$ if and only if $x-x'$ lies in the image of $\pi_1(\mathcal E_\bullet(P))\to\Omega^{\B}_{m+1}(P)$.
\end{lemma}
\begin{proof}
The argument for exactness at $\Omega^{\B}_{m+1}(P)$ in Lemma~\ref{exactB} applies verbatim: if $x=[N,g]$ and $x'=[N',g']$, then $\theta(x)=\theta(x')$ means that for some based loop $h$ in $\mathrm{Map}_\bullet(M,P)$ at $c$ the closed manifold $N\sqcup(-N')\sqcup M\times S^1$ with the maps $g$, $g'$ and $h$ is null-bordant over $P$, that is $x-x'=-[M\times S^1,h]$. The converse is clear.
\end{proof}

\subsection{The map $\gamma$}
\label{subsec:gamma}

Let $\xi\in\widehat\Omega^{\B}_{m+1}(P,M)$. By Lemma~\ref{lem:normal_form} choose a representative $(W,F)$ in standard form, with $\partial W=\partial_1W\sqcup N$, $\partial_1W=-M$, $F|_{\partial_1W}=c$ and $g=F|_N\colon N\to M$ of degree one. Apply Lemma~\ref{lem:surgery} with $A=N$, $C=\partial_1W$ and $d=m+1$. This replaces $(W,F)$ by a bordant triple rel boundary, hence the same class $\xi$, and we may assume $F_*$ is an isomorphism on $\pi_j$ for $j\le\ell_{m+1}$. By Lemma~\ref{lem:lifting} there is a lift $r\colon W\to M$ of $F$ extending $g$ on $N$. Put
\[
\gamma(\xi):=[\,r|_{\partial_1W}\,]\in[M,M]_\bullet .
\]

\begin{lemma}\label{gamma_well_defined}
The class $\gamma(\xi)$ is a well-defined element of $\Aut_\bullet(M)$, independent of the representative, the surgeries and the lift. Moreover $\pi(\gamma(\xi))=1$.
\end{lemma}
\begin{proof}
Put $f=r|_{\partial_1W}$. Since $c\circ f\simeq F|_{\partial_1W}=c$, the map $f$ induces isomorphisms on $\pi_i$ for $i\le n$, in particular on $\pi_1$, and $\pi(f)=1$ once we know $f$ is an equivalence. The triple $(W,r)$ is a $\B$-bordism over $M$ between $(N,g)$ and $(M,f)$, since $\partial W=-M\sqcup N$. So $f_*[M]=g_*[N]=[M]$, and $f$ has degree one. By Lemma~\ref{lem:PD_upgrade}, $f$ is a homotopy equivalence.

For independence, let $(W,F,r)$ and $(W',F',r')$ be two such choices for the same class. Lemma~\ref{lem:normal_form} gives a $\B$-manifold $T$ over $P$ with $\partial T=W\cup W'\cup(M\times I)\cup Y_N$, with $W$, $W'$ and $Y_N$ carrying the lifts $r$, $r'$ and $G_N$; these agree on $N$ and $N'$ by construction. Put $A=W\cup W'\cup Y_N$ and $C=M\times I$. After surgery in the interior of $T$ (Lemma~\ref{lem:surgery}; here $\dim T=m+2$ and $\Phi|_C=c\circ\mathrm{pr}_1$), Lemma~\ref{lem:lifting} gives a lift $\widehat\Phi\colon T\to M$ extending the given map on $A$. Its restriction to $C=M\times I$ is a based homotopy from $r|_{\partial_1W}$ to $r'|_{\partial_1W'}$. The same argument applies with $W=W'$ and different lifts $r,r'$, taking $T=W\times I$.
\end{proof}

\begin{lemma}\label{alpha_gamma}
$\alpha\circ\gamma=\partial\colon\widehat\Omega^{\B}_{m+1}(P,M)\to\widehat\Omega^{\B}_m(M)$.
\end{lemma}
\begin{proof}
With $(W,F,r)$ as in the definition of $\gamma$ we have $\partial\xi=[\partial W,F|_{\partial W}]=[N,g]-[M,\id]$. The lift $r$ gives a bordism over $M$ between $(N,g)$ and $(M,f)$ with $f=\gamma(\xi)$, so $[N,g]=[M,f]$ in $\Omega^{\B}_m(M)$ and $\alpha(\gamma(\xi))=[M,f]-[M,\id]=\partial\xi$.
\end{proof}

\begin{lemma}\label{invcom}
Let $j\colon\Omega^{\B}_{m+1}(P)\to\widehat\Omega^{\B}_{m+1}(P,M)$ be the map of Lemma~\ref{bordism exact}, and let $\rho\colon\Omega^{\B}_{m+1}(P)\to\Aut_\bullet(M)$ be the composite of $\theta$ with the map $\widetilde{\mathcal H}(M)\to\Aut_\bullet(M)$ of Lemma~\ref{exactA}. Then $\gamma(j(x))=\rho(x)^{-1}$ for all $x\in\Omega^{\B}_{m+1}(P)$.
\end{lemma}
\begin{proof}
Let $x=[N,g]$ and let $W_0$ be the interior connected sum of $M\times I$ and $N$ with the map $c\circ\mathrm{pr}_1\sharp g$. After surgery in the interior as in Lemma~\ref{lem:surgery} (with $C=M\times\{0\}$, equivalently $C=M\times\{1\}$) we obtain $(W,F)$ with $F_*$ an isomorphism on $\pi_j$ for $j\le\ell_{m+1}$, which represents $\theta(x)$ in $\widetilde{\mathcal H}(P)$ and, in standard form with $\partial_1W=M\times\{0\}$ and $N'=M\times\{1\}$, $g'=\id_M$, represents $j(x)$ in the relative group.

By the construction of the bijection in Lemma~\ref{HeM=HeP}, $\rho(x)=[\widehat f]$, where $\widehat F\colon W\to M$ is a lift of $F$ extending $\id_M$ on $M\times\{0\}$ (it exists by Lemma~\ref{lem:lifting}) and $\widehat f=\widehat F|_{M\times1}$. Since $F|_{M\times1}=c$, we have $c\circ\widehat f\simeq c$, so $\widehat f$ is a homotopy equivalence in $\ker\pi$, and $c\circ\widehat f^{-1}\simeq c\circ\widehat f\circ\widehat f^{-1}\simeq c$. Put $r_0=\widehat f^{-1}\circ\widehat F$. Then $c\circ r_0\simeq c\circ\widehat F\simeq F$, and $r_0|_{M\times1}=\widehat f^{-1}\circ\widehat f\simeq\id_M$. Using the homotopy extension property we change $r_0$ by a homotopy so that it restricts to $\id_M$ on $M\times\{1\}$. The resulting homotopy $c\circ r_0\simeq F$ restricts on $M\times\{1\}$ to a loop in $\mathrm{Map}(M,P)$ based at $c$. Since $\widehat f$ is a homotopy equivalence, precomposition with $\widehat f^{-1}$ is a bijection on $\pi_1$ of the mapping space, so changing the homotopy $c\circ\widehat f^{-1}\simeq c$ by a suitable loop makes this loop trivial. Then $r_0$ is a lift of $F$ extending $\id_M$ on $M\times\{1\}$ in the sense of the definition, and by Lemma~\ref{gamma_well_defined} the class $\gamma(j(x))$ may be computed with it. Hence $\gamma(j(x))=[r_0|_{M\times0}]=[\widehat f^{-1}]=\rho(x)^{-1}$.
\end{proof}

\begin{corollary}\label{im}
The sequence of pointed sets
\[
\pi_1(\mathcal E_\bullet(P))\xrightarrow{\ \ j\circ\kappa\ }\widehat\Omega^{\B}_{m+1}(P,M)\xrightarrow{\ \gamma\ }\Aut_\bullet(M)\xrightarrow{\ \pi\ }\Aut_\bullet(P)
\]
is exact, where $\kappa\colon\pi_1(\mathcal E_\bullet(P))\to\Omega^{\B}_{m+1}(P)$ is the map of Lemma~\ref{exactB}. In particular $\ker\pi=\operatorname{im}\gamma$.
\end{corollary}
\begin{proof}
\emph{Exactness at $\Aut_\bullet(M)$.} By Lemma~\ref{gamma_well_defined}, $\operatorname{im}\gamma\subset\ker\pi$. Let $f\in\ker\pi$, so there is a based homotopy $h\colon M\times I\to P$ from $c$ to $c\circ f$. Then $(M\times I,h)$ is in standard form with $\partial_1=M\times\{0\}$, $N=M\times\{1\}$ and $g=f$, which has degree one, so it defines a class $\xi\in\widehat\Omega^{\B}_{m+1}(P,M)$. Lemma~\ref{lem:lifting}, with $A=N$ and $C=\partial_1$ (note $h\simeq c\circ\mathrm{pr}_1$ is an isomorphism on $\pi_j$ for $j\le n$, so no surgery is needed), gives a lift $r\colon M\times I\to M$ of $h$ with $r|_{M\times1}=f$. Since $c\circ r\simeq h$, the map $r$ is a homotopy in $M$ from $r|_{M\times0}$ to $f$, so $\gamma(\xi)=[r|_{M\times0}]=[f]$.

\emph{Exactness at $\widehat\Omega^{\B}_{m+1}(P,M)$.} If $\xi=j(\kappa(\lambda))$, then $\theta(\kappa(\lambda))$ is the base point by Lemma~\ref{exactB}, so $\rho(\kappa(\lambda))=\id$ and $\gamma(\xi)=\id$ by Lemma~\ref{invcom}. Conversely let $\gamma(\xi)=\id$. By Lemma~\ref{alpha_gamma}, $\partial\xi=\alpha(\id)=0$, so $\xi=j(x)$ for some $x\in\Omega^{\B}_{m+1}(P)$ by Lemma~\ref{bordism exact}, and $\rho(x)=\id$ by Lemma~\ref{invcom}. Thus $\theta(x)\in\widetilde{\mathcal H}(M)$ maps to $\id$ in $\Aut_\bullet(M)$, so $\theta(x)=\iota(z)=\theta(c_*z)$ for some $z\in\Omega^{\B}_{m+1}(M)$ by Lemma~\ref{exactA}. By Lemma~\ref{lem:fibres}, $x-c_*z=\kappa(\lambda)$ for some $\lambda$, hence $\xi=j(x)=j(\kappa(\lambda))+j(c_*z)=j(\kappa(\lambda))$, since $j\circ c_*=0$ by Lemma~\ref{bordism exact}.
\end{proof}

\begin{theorem}[Theorem A]\label{thm:braid_exists}
Let $M$ be a closed, connected, oriented $\B$-manifold of dimension $m\in\{2n-1,2n\}$, $n\ge2$, $m\ge4$, whose fundamental group is finite or infinite cyclic, and assume (T). Then the following four sequences of pointed sets are exact:
\begin{align*}
&\Omega^{\B}_{m+1}(M)\to\Omega^{\B}_{m+1}(P)\to\widehat\Omega^{\B}_{m+1}(P,M)\to\widehat\Omega^{\B}_m(M)\to\Omega^{\B}_m(P),\\
&\Omega^{\B}_{m+1}(M)\to\widetilde{\mathcal H}(M)\to\Aut_\bullet(M)\xrightarrow{\alpha}\widehat\Omega^{\B}_m(M),\\
&\pi_1(\mathcal E_\bullet(P))\to\Omega^{\B}_{m+1}(P)\to\widetilde{\mathcal H}(M)\to\Aut_\bullet(P)\xrightarrow{\beta}\Omega^{\B}_m(P),\\
&\pi_1(\mathcal E_\bullet(P))\to\widehat\Omega^{\B}_{m+1}(P,M)\xrightarrow{\gamma}\Aut_\bullet(M)\xrightarrow{\pi}\Aut_\bullet(P).
\end{align*}
The diagram of Theorem~A commutes, except that the two composites $\Omega^{\B}_{m+1}(P)\to\Aut_\bullet(M)$ agree only up to inversion in $\Aut_\bullet(M)$.
\end{theorem}
\begin{proof}
The four sequences are Lemma~\ref{bordism exact}, Lemma~\ref{exactA}, Lemma~\ref{exactB} and Corollary~\ref{im}. For the commutativity: the diagram has the following squares. (i) $\Omega^{\B}_{m+1}(M)\to\Omega^{\B}_{m+1}(P)\to\widetilde{\mathcal H}(M)$ equals $\iota$, since $\theta\circ c_*=\iota$. (ii) $\widetilde{\mathcal H}(M)\to\Aut_\bullet(M)\xrightarrow{\pi}\Aut_\bullet(P)$ equals $\widetilde{\mathcal H}(M)\xrightarrow{c_\sharp}\widetilde{\mathcal H}(P)\to\Aut_\bullet(P)$, since $\phi_f\circ c\simeq c\circ f$. (iii) $\Aut_\bullet(M)\xrightarrow{\alpha}\widehat\Omega^{\B}_m(M)\to\Omega^{\B}_m(P)$ equals $\beta\circ\pi$, since $\beta(\phi_f)=[M,\phi_fc]-[M,c]=[M,cf]-[M,c]=c_*\alpha(f)$. (iv) $\alpha\circ\gamma=\partial$ by Lemma~\ref{alpha_gamma}. (v) The two maps $\Omega^{\B}_{m+1}(P)\to\Aut_\bullet(M)$ are inverse to each other by Lemma~\ref{invcom}. The remaining triangles hold by definition.
\end{proof}

\begin{remark}\label{rem:image}
The image of $\pi$ is contained in the stabilizer of $c_*[M]\in H_m(P;\mathbb{Z})$, since $\phi_f\circ c\simeq c\circ f$ and $f_*[M]=[M]$. The braid also describes part of the image: by exactness of the third sequence in Theorem~\ref{thm:braid_exists} and Lemma~\ref{HeM=HeP}, $\ker\beta$ is the image of $\widetilde{\mathcal H}(M)\to\Aut_\bullet(P)$, that is $\ker\beta=\pi(\ker\alpha)$. Here $\ker\beta$ can be a proper subset of $\operatorname{im}\pi$, since $\beta(\pi(f))=c_*\alpha(f)$ need not vanish for $f\notin\ker\alpha$.
\end{remark}
\section{Examples: products and connected sums of spheres}
\label{sec:applications}

All manifolds in this section are products or connected sums of spheres. They are stably parallelizable, hence $\B$-manifolds for every $n$ (we take the $\B$-structure induced by a stable framing), and simply connected except $S^1\times S^{m-1}$, where (T) holds because $c$ is a trivial fibration. Bordism groups of products of spheres are computed from the stable splitting $\Sigma^\infty(A\times B)_+\simeq S^0\vee\Sigma^\infty A\vee\Sigma^\infty B\vee\Sigma^\infty(A\wedge B)$, which splits any generalized homology theory.

\subsection{Pinch-fold twists and maps out of a product of two spheres}
\label{subsec:twists}

The computations for $S^3\times S^3$, for its connected sums, for $S^3\times S^4$ and for $S^4\times S^4$ all use the same family of self-maps. Let $X$ be a closed manifold of dimension $m\ge2$, or more generally a connected CW complex with a single top cell of dimension $m$ that is attached to the rest by a map. Let $\mathrm{pinch}\colon X\to X\vee S^m$ be the map that collapses the boundary of a small disc $D$ contained in the open top cell, and let $q\colon X\to S^m$ be the composite of $\mathrm{pinch}$ with the projection to $S^m$. For $\tilde\kappa\in\pi_m(X)$ the \emph{pinch-fold map} is
\[
t_{\tilde\kappa}:=\nabla\circ(\id\vee\tilde\kappa)\circ\mathrm{pinch}\colon X\to X .
\]
More generally, $\pi_m(Y)$ acts on $[X,Y]_\bullet$ for every space $Y$ by $a\cdot\varphi=\nabla\circ(\varphi\vee a)\circ\mathrm{pinch}$, and $\varphi\circ t_{\tilde\kappa}=(\varphi\circ\tilde\kappa)\cdot\varphi$. If $X$ is simply connected, or if $\pi_1(X)$ acts trivially on $\pi_m(X)$, the class of $t_{\tilde\kappa}$ does not depend on the position of the disc. If the Hurewicz image of $\tilde\kappa$ in $H_m(X;\mathbb{Z})$ vanishes and $X$ is simply connected, then $t_{\tilde\kappa}$ induces the identity on homology, hence is a homotopy equivalence of degree one.

\begin{lemma}\label{lem:coaction}
\begin{enumerate}
\item[(a)] Let $G$ be a topological group, $\varphi\colon X\to G$ and $\psi\colon S^m\to G$ based maps. Then $\nabla\circ(\varphi\vee\psi)\circ\mathrm{pinch}\simeq\varphi\cdot(\psi\circ q)$ (pointwise product in $[X,G]_\bullet$). Consequently $\varphi\circ t_{\tilde\kappa}\simeq\varphi\cdot q^*(\varphi\circ\tilde\kappa)$.
\item[(b)] If $q\circ\tilde\kappa\simeq*$, then $q\circ t_{\tilde\kappa}\simeq q$. Let $K\subset X$ be the complement of the open top cell. If $\tilde\kappa$ and $\tilde\kappa'$ are represented by maps into $K$, then $t_{\tilde\kappa}$ and $t_{\tilde\kappa'}$ restrict to the inclusion on $K$ and $t_{\tilde\kappa}\circ t_{\tilde\kappa'}\simeq t_{\tilde\kappa+\tilde\kappa'}\simeq t_{\tilde\kappa'}\circ t_{\tilde\kappa}$.
\item[(c)] If $h\colon X\to X$ is an orientation-preserving diffeomorphism, then $h\circ t_{\tilde\kappa}\circ h^{-1}\simeq t_{h_*\tilde\kappa}$.
\end{enumerate}
\end{lemma}
\begin{proof}
(a) Let $j\colon G\vee G\to G\times G$ be the inclusion and $\mu$ the multiplication, so $\mu\circ j=\nabla$. The map $j\circ(\varphi\vee\psi)\circ\mathrm{pinch}$ has components $(\varphi\circ\mathrm{pr}_X\circ\mathrm{pinch},\ \psi\circ q)\simeq(\varphi,\psi\circ q)$, because $\mathrm{pr}_X\circ\mathrm{pinch}\simeq\id_X$. Apply $\mu$. Taking $\psi=\varphi\circ\tilde\kappa$ gives the second statement.

(b) The first claim holds since $q\circ t_{\tilde\kappa}=\nabla(q\vee q\tilde\kappa)\circ\mathrm{pinch}\simeq q\circ\mathrm{pr}_X\circ\mathrm{pinch}\simeq q$. For the second, $t_{\tilde\kappa}\circ\tilde\kappa'\simeq\tilde\kappa'$, so $t_{\tilde\kappa}\circ t_{\tilde\kappa'}\simeq\nabla(t_{\tilde\kappa}\vee\tilde\kappa')\circ\mathrm{pinch}$. This is the map that collapses the boundaries of two disjoint discs $D_1,D_2$ in the top cell and folds in $\tilde\kappa'$ on $D_1$ and $\tilde\kappa$ on $D_2$. For $m\ge2$ an isotopy of the open top cell exchanges $D_1$ and $D_2$, so the composite is symmetric in $\tilde\kappa,\tilde\kappa'$. If $D\supset D_1\cup D_2$ is a disc, collapsing the boundaries of $D_1$ and $D_2$ inside $D$ is the comultiplication $S^m\to S^m\vee S^m$ applied to the sphere obtained by collapsing $\partial D$. So the two-disc map is the single-disc map for $\nabla(\tilde\kappa\vee\tilde\kappa')\circ(\text{comultiplication})=\tilde\kappa+\tilde\kappa'$.

(c) Since $h$ is a diffeomorphism, $\mathrm{pinch}_D\circ h^{-1}=(h^{-1}\vee\id)\circ\mathrm{pinch}_{h(D)}$, so $h\,t_{\tilde\kappa}\,h^{-1}=\nabla(\id\vee h\tilde\kappa)\circ\mathrm{pinch}_{h(D)}$. An orientation-preserving embedded disc in a connected manifold is isotopic to any other, so this is homotopic to $t_{h_*\tilde\kappa}$.
\end{proof}

We also need to count the maps $S^p\times S^q\to Y$ with prescribed restriction to $S^p\vee S^q$.

\begin{lemma}\label{lem:orbit}
Let $p,q\ge2$, $M=S^p\times S^q$, $Y$ a space, and $g\in\pi_q(Y)$. Let $\varphi_0=g\circ\mathrm{pr}_2\colon M\to Y$. Then the fibre of the restriction map $[M,Y]_\bullet\to[S^p\vee S^q,Y]_\bullet$ over the class of $\varphi_0|_{S^p\vee S^q}=(*,g)$ is the orbit of $\varphi_0$ under $\pi_{p+q}(Y)$, and the stabilizer of $\varphi_0$ is the subgroup
\[
[\pi_{p+1}(Y),g]:=\{[a,g]\mid a\in\pi_{p+1}(Y)\}\subset\pi_{p+q}(Y).
\]
Hence the fibre is in bijection with $\pi_{p+q}(Y)/[\pi_{p+1}(Y),g]$. The same holds with the factors interchanged: for $g\in\pi_p(Y)$ and $\varphi_0=g\circ\mathrm{pr}_1$ the stabilizer is $[\pi_{q+1}(Y),g]$.
\end{lemma}
\begin{proof}
Write $K=S^p\vee S^q$, so $M=K\cup_we^{p+q}$ with $w=[\iota_p,\iota_q]$, and consider the cofibre sequence $S^{p+q-1}\xrightarrow{w}K\to M\xrightarrow{q}S^{p+q}$. The fibration sequence $\mathrm{Map}_\bullet(S^{p+q},Y)\to\mathrm{Map}_\bullet(M,Y)\to\mathrm{Map}_\bullet(K,Y)\xrightarrow{w^*}\mathrm{Map}_\bullet(S^{p+q-1},Y)$ shows that the fibre of the restriction map over $[(*,g)]$ is the quotient of $\pi_{p+q}(Y)$ by the action of $\pi_1(\mathrm{Map}_\bullet(K,Y),(*,g))$, and that $\pi_{p+q}(Y)$ acts on the fibre through the coaction, by the definition of the action, since $w^*(*,g)=[*,g]=0$.

It remains to identify the stabilizer. Maps $\varphi\colon M\to Y$ with $\varphi|_{K}=(*,g)$ strictly are adjoint to based maps $S^p\to F_g$, where $F_g=\mathrm{Map}_\bullet(S^q,Y)_g$ is the component of $g$; under the translation $F_g\simeq\Omega^qY$ they are elements of $\pi_p(\Omega^qY)=\pi_{p+q}(Y)$, and the element $b$ corresponds to $\varphi_0+b\circ q$. The group $\pi_1(\mathrm{Map}_\bullet(K,Y),(*,g))=\pi_{p+1}(Y)\times\pi_1(F_g,g)$ acts on $\pi_p(F_g,g)$. The factor $\pi_1(F_g,g)=\pi_{q+1}(Y)$ acts trivially, because $F_g\simeq\Omega^qY$ is a loop space and $\pi_1$ acts trivially on the homotopy groups of an $H$-group. The factor $\pi_{p+1}(Y)$ acts by translation by the boundary map of the evaluation fibration $F_g\to\mathrm{Map}(S^q,Y)_g\to Y$, which by G.~W.~Whitehead \cite{whitehead46} is $a\mapsto\pm[a,g]$. Hence the stabilizer is $[\pi_{p+1}(Y),g]$. The statement with the factors interchanged follows by symmetry.
\end{proof}

\subsection{The case $M=S^1\times S^{m-1}$}
\label{subsec:S1xSm}

Let $M=S^1\times S^{m-1}$ with $m\ge4$ and put $n=\lceil m/2\rceil$. Then $n\ge2$ and $m-1\ge n+1$, so $\pi_1(M)=\mathbb{Z}$ and $\pi_i(M)=\pi_i(S^{m-1})=0$ for $2\le i\le n$. Hence $P\simeq S^1$, and $c\colon M\to S^1$ is the projection, which is a trivial fibration, so (T) holds. Theorem~\ref{thm:braid_exists} applies. Since $n\ge2$ the spectrum $M\B$ agrees with the sphere spectrum through dimension $1$, so $\Omega_1=\pi_1^s\cong\mathbb{Z}/2$, generated by the circle with its Lie group framing.

Write $\Omega_k=\Omega^{\B}_k(*)$. The stable splitting $\Sigma^\infty M_+\simeq S^0\vee S^1\vee S^{m-1}\vee S^m$ gives
\[
\Omega^{\B}_k(M)\cong\Omega_k\oplus\Omega_{k-1}\oplus\Omega_{k-m+1}\oplus\Omega_{k-m},\qquad
\Omega^{\B}_k(S^1)\cong\Omega_k\oplus\Omega_{k-1},
\]
and $c_*$ is the projection onto the first two summands. The degree is the projection to $\Omega_0=\mathbb{Z}$ in $\Omega^{\B}_m(M)$. Hence $\widehat\Omega^{\B}_m(M)\cong\Omega_m\oplus\Omega_{m-1}\oplus\Omega_1$ and
\[
\ker\bigl(c_*\colon\widehat\Omega^{\B}_m(M)\to\Omega^{\B}_m(S^1)\bigr)\cong\Omega_1\cong\mathbb{Z}/2 .
\]
Since $c$ has a section, $\Omega^{\B}_{m+1}(M)\to\Omega^{\B}_{m+1}(S^1)$ is onto, so by Lemma~\ref{bordism exact}
\[
\partial\colon\widehat\Omega^{\B}_{m+1}(S^1,M)\xrightarrow{\ \cong\ }\ker\bigl(\widehat\Omega^{\B}_m(M)\to\Omega^{\B}_m(S^1)\bigr)\cong\mathbb{Z}/2 .
\]

\begin{theorem}\label{splitting}
For $M=S^1\times S^{m-1}$ with $m\ge4$ we have $\Aut_\bullet(M)\cong\mathbb{Z}/2\oplus\mathbb{Z}/2$.
\end{theorem}
\begin{proof}
By Corollary~\ref{im} there is an exact sequence $\widehat\Omega^{\B}_{m+1}(S^1,M)\xrightarrow{\gamma}\Aut_\bullet(M)\xrightarrow{\pi}\Aut_\bullet(S^1)$, where $\Aut_\bullet(S^1)\cong\mathbb{Z}/2$ is generated by complex conjugation. So $\ker\pi=\operatorname{im}\gamma$ has at most two elements. By Lemma~\ref{alpha_gamma}, $\alpha(\gamma(w))=\partial w$. If $w\ne0$ then $\partial w\ne0$, whereas $\alpha(\id)=0$, so $\gamma(w)\ne\id$. Thus $\ker\pi$ has order $2$; let $\tau$ be its non-trivial element. A normal subgroup of order $2$ is central, so $\tau$ is central.

Let $r\colon S^{m-1}\to S^{m-1}$ be a reflection and $s(x,y)=(\bar x,r(y))$. It is an orientation-preserving based self-diffeomorphism of order $2$ with $\pi(s)\ne1$, so $\pi$ is onto and $|\Aut_\bullet(M)|=4$. Since $\tau$ is central, $(s\tau)^2=s^2\tau^2=1$, so every non-trivial element has order $2$, and $\Aut_\bullet(M)\cong(\mathbb{Z}/2)^2$.
\end{proof}

\begin{remark}\label{rem:generators}
Let $\Aut(M)$ be the group of all based self-equivalences, including those of degree $-1$. It contains the commuting involutions $\psi(x,y)=(\bar x,y)$ and $\rho(x,y)=(x,r(y))$, both of degree $-1$, and $\tau$. The subgroup $\Aut_\bullet(M)$ has index $2$, so $|\Aut(M)|=8$. The kernel of $\Aut(M)\to\Aut(S^1)$ is normal and contains $\rho$ and $\tau$; it has order $4$. As $\tau$ generates the unique subgroup of order $2$ of the normal subgroup $\ker\pi$ of $\Aut_\bullet(M)$, which is itself normal in $\Aut(M)$ (it is $\ker\pi\cap\Aut_\bullet(M)$), $\tau$ is central in $\Aut(M)$. Hence $\psi,\rho,\tau$ commute and $\Aut(M)\cong(\mathbb{Z}/2)^3$. This agrees with Sieradski's computation for $S^1\times S^3$ and $S^1\times S^7$ \cite{sieradski}.

The element $\tau$ is a pinch-fold map. Let $\eta^{m-3}\colon S^m\to S^{m-1}$ be the $(m-3)$-fold suspension of the Hopf map $\eta\colon S^3\to S^2$, let $j\colon S^{m-1}\to M$ be the inclusion of the fibre, and $\tilde\kappa=j\circ\eta^{m-3}\in\pi_m(M)$. The fundamental group acts trivially on $\pi_m(M)=\pi_m(S^{m-1})$, and the Hurewicz image of $\tilde\kappa$ is zero, so $t_{\tilde\kappa}$ has degree one and induces the identity on $\pi_1$. By Lemma~\ref{lem:PD_upgrade} it is a homotopy equivalence, with $\pi(t_{\tilde\kappa})=1$ and $\alpha(t_{\tilde\kappa})=[S^m,j\eta^{m-3}]$. Under the splitting above this class is the image of the generator of $\pi_m(S^{m-1})\cong\pi_1^s$ in the summand $\Omega_1=\mathbb{Z}/2$, so it is non-zero. Hence $t_{\tilde\kappa}=\tau$.
\end{remark}
\subsection{The case $M = S^3 \times S^3$}
\label{subsec:S3xS3}

Let $M=S^3\times S^3$. This is a parallelizable $6$-manifold, hence spin, that is a $BO\langle4\rangle$-manifold with $n=3$ and $m=2n$. Since $\pi_1(M)=\pi_2(M)=0$ and $\pi_3(M)\cong\mathbb{Z}^2$, its Postnikov $3$-type is $P=K(\mathbb{Z}^2,3)$ and $\Aut_\bullet(P)=GL_2(\mathbb{Z})$. Hypothesis (T) holds because $M$ is simply connected, so Theorem~\ref{thm:braid_exists} applies with $\B=BSpin$. We first compute what the braid gives, namely a bound for the relative bordism group, and then determine $\Aut_\bullet(M)$ by an elementary argument that identifies the kernel of $\pi$ and the action of the quotient on it.

\begin{lemma}\label{lem:homology_K(ZZ,3)}
The integral homology of $P = K(\mathbb{Z}^2, 3)$ in degrees $\le 8$ is
\[
  H_i(P) \cong \begin{cases}
  \mathbb{Z} & i = 0,\\
  0 & i = 1, 2, 4,\\
  \mathbb{Z}^2 & i = 3, \\
  (\mathbb{Z}/2)^2 & i = 5, \\
  \mathbb{Z} & i = 6, \\
  (\mathbb{Z}/3)^2 & i = 7, \\
  (\mathbb{Z}/2)^4 & i = 8.
  \end{cases}
\]
In terms of the two fundamental classes $\iota_1,\iota_2\in H^3(P;\mathbb{Z}/2)$, the groups $H^k(P;\mathbb{Z}/2)$ for $k=3,5,6,8$ have bases $\{\iota_a\}$, $\{\mathrm{Sq}^2\iota_a\}$, $\{\iota_1^2,\iota_1\iota_2,\iota_2^2\}$ and $\{\iota_a\,\mathrm{Sq}^2\iota_b\}$ respectively, and $H^4(P;\mathbb{Z}/2)=H^7(P;\mathbb{Z}/2)=0$.
\end{lemma}
\begin{proof}
The integral homology of $K(\mathbb{Z},3)$ in degrees $0,\dots,8$ is $\mathbb{Z},0,0,\mathbb{Z},0,\mathbb{Z}/2,0,\mathbb{Z}/3,\mathbb{Z}/2$; this is the classical computation from the Serre spectral sequence of the path fibration $K(\mathbb{Z},2)\to *\to K(\mathbb{Z},3)$ (see \cite[Chapter~6]{moshertangora} or \cite[Section~4.L]{hatcher}). The stated groups follow from the K\"unneth formula for $P=K(\mathbb{Z},3)\times K(\mathbb{Z},3)$, noting that $H_3$ is free so the Tor terms vanish in this range. The mod $2$ cohomology of $K(\mathbb{Z},3)$ is a polynomial algebra on $\mathrm{Sq}^I\iota_3$ for admissible $I$ of excess less than $3$ not ending in $1$, which in degrees $\le 8$ gives the classes $\iota_3$, $\mathrm{Sq}^2\iota_3$, and their products. The statement for $P$ follows from the K\"unneth formula.
\end{proof}

\begin{theorem}\label{thm:bordism_S3S3}
Let $M = S^3 \times S^3$ and $P = K(\mathbb{Z}^2, 3)$. Then
\begin{itemize}
\item[(i)] $\widehat{\Omega}^{\mathrm{spin}}_6(M) = 0$, and
\item[(ii)] $\widehat{\Omega}^{\mathrm{spin}}_7(P,M)$ is a finite abelian group of order at most $144=2^4\cdot 3^2$.
\end{itemize}
\end{theorem}

\begin{proof}
We use the spin bordism groups of a point in low degrees \cite{abp}, namely $\Omega^{\mathrm{spin}}_q\cong\mathbb{Z},\ \mathbb{Z}/2,\ \mathbb{Z}/2,\ 0,\ \mathbb{Z},\ 0,\ 0,\ 0$ for $q=0,\dots,7$, and the two lowest differentials of the AHSS. By \cite{teichner}, for a space $X$ the differential $d_2\colon E^2_{p,0}=H_p(X;\mathbb{Z})\to E^2_{p-2,1}=H_{p-2}(X;\mathbb{Z}/2)$ is the composite of reduction mod $2$ with the dual of $\mathrm{Sq}^2$, and $d_2\colon E^2_{p,1}\to E^2_{p-2,2}$ is the dual of $\mathrm{Sq}^2$.

\smallskip
\noindent\textbf{Part (i).} The stable splitting $\Sigma^\infty M_+\simeq S^0\vee S^3\vee S^3\vee S^6$ gives $\Omega^{\mathrm{spin}}_k(M)\cong\Omega^{\mathrm{spin}}_k\oplus(\Omega^{\mathrm{spin}}_{k-3})^2\oplus\Omega^{\mathrm{spin}}_{k-6}$. For $k=6$ this is $0\oplus0\oplus\mathbb{Z}\cong\mathbb{Z}$, and the degree map is an isomorphism onto $\mathbb{Z}$. Hence $\widehat{\Omega}^{\mathrm{spin}}_6(M)=\ker(\deg)=0$. For $k=7$ the splitting gives $\Omega^{\mathrm{spin}}_7(M)\cong\mathbb{Z}^2\oplus\mathbb{Z}/2$, where $\mathbb{Z}^2=\Omega_4^{\mathrm{spin}}\otimes H_3(M)$ is the bottom filtration piece of the AHSS.

\smallskip
\noindent\textbf{Part (ii).} By Lemma~\ref{bordism exact} and Part~(i), $\partial=0$ and $\widehat{\Omega}^{\mathrm{spin}}_7(P,M)\cong\operatorname{coker}\bigl[c_*\colon\Omega^{\mathrm{spin}}_7(M)\to\Omega^{\mathrm{spin}}_7(P)\bigr]$. We use the AHSS for $P$ in total degree $7$. Let $F_p\Omega^{\mathrm{spin}}_7(P)$ denote the image of $\Omega^{\mathrm{spin}}_7(P^{(p)})\to\Omega^{\mathrm{spin}}_7(P)$, where $P^{(p)}$ is the $p$-skeleton, so that $F_p\Omega^{\mathrm{spin}}_7(P)/F_{p-1}\Omega^{\mathrm{spin}}_7(P)\cong E^\infty_{p,7-p}$. By Lemma~\ref{lem:homology_K(ZZ,3)} and universal coefficients, $H_6(P;\mathbb{Z}/2)\cong(\mathbb{Z}/2)^3$, $H_5(P;\mathbb{Z}/2)\cong(\mathbb{Z}/2)^2$, $H_4(P;\mathbb{Z}/2)=0$ and $H_7(P;\mathbb{Z}/2)=0$. The terms of total degree $7$ are
\[
E^2_{7,0}\cong(\mathbb{Z}/3)^2,\quad E^2_{6,1}\cong(\mathbb{Z}/2)^3,\quad E^2_{5,2}\cong(\mathbb{Z}/2)^2,\quad E^2_{4,3}=0,\quad E^2_{3,4}\cong\mathbb{Z}^2,
\]
and $E^2_{p,q}=0$ for $p\le2$ in this total degree.

The differential $d_2\colon E^2_{8,0}\to E^2_{6,1}$ has rank one. Indeed, reduction mod $2$ is an isomorphism $H_8(P;\mathbb{Z})\cong(\mathbb{Z}/2)^4\to H_8(P;\mathbb{Z}/2)$, because $H_7(P;\mathbb{Z})$ has no $2$-torsion, and $d_2$ is then dual to $\mathrm{Sq}^2\colon H^6(P;\mathbb{Z}/2)\to H^8(P;\mathbb{Z}/2)$. Since $\mathrm{Sq}^1\iota_a=0$, the Cartan formula gives $\mathrm{Sq}^2(\iota_a^2)=(\mathrm{Sq}^1\iota_a)^2=0$ and
\[
\mathrm{Sq}^2(\iota_1\iota_2)=\mathrm{Sq}^2\iota_1\cdot\iota_2+\iota_1\cdot\mathrm{Sq}^2\iota_2\ne0,
\]
so $\mathrm{Sq}^2$ has rank one. Hence $E^3_{6,1}$ has order $4$. The term $E^2_{6,1}$ supports no non-zero differential, since all possible targets $E^2_{4,2}$, $E^2_{3,3}$, $E^2_{2,4}$, $E^2_{1,5}$, $E^2_{0,6}$ vanish. Also $E^2_{5,2}$ has order $4$ and $E^2_{7,0}$ has order $9$. Since $E^2_{4,3}=0$, we have $F_4\Omega^{\mathrm{spin}}_7(P)=F_3\Omega^{\mathrm{spin}}_7(P)$, and the successive quotients of $\Omega^{\mathrm{spin}}_7(P)/F_3\Omega^{\mathrm{spin}}_7(P)$ are $E^\infty_{5,2}$, $E^\infty_{6,1}$ and $E^\infty_{7,0}$. Hence this quotient has order at most $4\cdot4\cdot9=144$.

Choose the cell structure of $P$ so that $P=M\cup(\text{cells of dimension}\ge5)$, hence $P^{(3)}=M^{(3)}=S^3\vee S^3\subset M$. Then $F_3\Omega^{\mathrm{spin}}_7(P)$ is the image of $\Omega^{\mathrm{spin}}_7(S^3\vee S^3)\to\Omega^{\mathrm{spin}}_7(P)$, which factors through $c_*$, so $F_3\Omega^{\mathrm{spin}}_7(P)$ lies in the image of $c_*$. Therefore the cokernel of $c_*$ is a quotient of $\Omega^{\mathrm{spin}}_7(P)/F_3\Omega^{\mathrm{spin}}_7(P)$, and has order at most $144$.
\end{proof}

\begin{remark}
A first reading of the AHSS suggests that $E^2_{6,1}$ contributes $(\mathbb{Z}/2)^3$, which would give a bound of $288$. The differential from $E^2_{8,0}$ kills the reduction mod $2$ of the class $c_*[M]$.
\end{remark}

The intersection form $s_M$ on $\pi_3(M) \cong \mathbb{Z} \oplus \mathbb{Z}$ is skew-symmetric, not symmetric: for $u, v \in H^3(M;\mathbb{Z})$, graded-commutativity of cup product gives $u \smile v = (-1)^{3\cdot 3} v \smile u = -v \smile u$, since $3 \cdot 3 = 9$ is odd.  In the basis $e_1, e_2 \in H_3(M;\mathbb{Z})$ dual to the two $S^3$ factors, with $s_M(e_1,e_2)=1$, this forces $s_M(e_2,e_1)=-1$ and $s_M(e_i,e_i)=0$, i.e.
\[
  s_M = \begin{pmatrix} 0 & 1 \\ -1 & 0 \end{pmatrix}.
\]
For a $2\times 2$ integer matrix $A$, $A^T s_M A = \det(A) \cdot s_M$ always, so
\[
  \Isom[\pi_3(M), s_M] \cong Sp_2(\mathbb{Z}) = SL_2(\mathbb{Z}),
\]
the subgroup of $GL_2(\mathbb{Z})$ of matrices of determinant $1$. This is consistent with $\Aut_\bullet(M)$ consisting of orientation-preserving equivalences: an automorphism of $\pi_3(M)$ of determinant $-1$ sends $s_M$ to $-s_M$, so it is not induced by an orientation-preserving self-map. Every self-equivalence induces an isometry, so the image of $\pi\colon\Aut_\bullet(M)\to\Aut_\bullet(P)=GL_2(\mathbb{Z})$ lies in $SL_2(\mathbb{Z})$.
The group $\Aut_\bullet(S^3\times S^3)$ can be described completely in elementary terms, because $S^3=SU(2)$ is a topological group. The argument uses the Samelson product $\langle\iota_3,\iota_3\rangle$, whose order is computed by Bott \cite{bott}.

\begin{lemma}\label{lem:AutN}
Let $M=S^3\times S^3$, let $N:=[M,S^3]_\bullet$, a group under pointwise multiplication in $S^3$, and let $x,y\in N$ be the two coordinate projections.
\begin{enumerate}
\item[(a)] The element $c:=[x,y]$ is central in $N$ and has order $12$, $N/\langle c\rangle\cong\mathbb{Z}^2$ is generated by the images of $x$ and $y$, and
\[
  N=\langle x,y,c\mid c\ \text{central},\ c^{12}=1,\ [x,y]=c\rangle .
\]
\item[(b)] The assignment $f\mapsto f^*$, $f^*(\varphi)=\varphi\circ f$ for a based map $\varphi\colon M\to S^3$ representing an element of $N$, is an anti-isomorphism from $\Aut_\bullet(M)$ onto the group $\Aut^+(N)$ of automorphisms of $N$ that induce a matrix of determinant $1$ on $N/\langle c\rangle\cong\mathbb{Z}^2$. In particular $\Aut_\bullet(M)\cong\Aut^+(N)$ as groups. The matrix induced by $f^*$ on $N/\langle c\rangle=H^3(M;\mathbb{Z})$ is the transpose of $\pi(f)$ in the dual basis.
\end{enumerate}
\end{lemma}
\begin{proof}
(a) Write $M=(S^3\vee S^3)\cup_w e^6$ with $w=[\iota_1,\iota_2]$, and let $q\colon M\to S^6$ be the pinch map. The Puppe sequence gives an exact sequence of groups
\[
  \pi_4(S^3)^2\xrightarrow{(\Sigma w)^*}\pi_6(S^3)\xrightarrow{q^*}N\xrightarrow{i^*}\pi_3(S^3)^2=\mathbb{Z}^2 .
\]
Since the suspension of a Whitehead product is null, $\Sigma w\simeq0$, so $q^*$ is injective and its image is a copy of $\pi_6(S^3)\cong\mathbb{Z}/12$. The map $i^*$ is onto because $i^*(x),i^*(y)$ form a basis.

The commutator of $g,h\in N$ is the composite of the diagonal $M\to M\times M$, the map $g\times h$, and the commutator map $S^3\times S^3\to S^3$. The commutator map is trivial on $S^3\vee S^3$, so $[g,h]$ factors as
\[
  M\xrightarrow{\ \bar\Delta\ }M\wedge M\xrightarrow{\ g\wedge h\ }S^3\wedge S^3\xrightarrow{\ \langle\iota_3,\iota_3\rangle\ }S^3,
\]
where $\bar\Delta$ is the reduced diagonal and $\langle\iota_3,\iota_3\rangle\in\pi_6(S^3)$ is the Samelson product. Since $\dim M=6$, the map $\bar\Delta$ may be deformed into the $6$-skeleton $(S^3\vee S^3)\wedge(S^3\vee S^3)$ of $M\wedge M$. If $h=q^*\alpha$, then $h$ is trivial on $S^3\vee S^3$, so $g\wedge h$ is trivial on this skeleton and $[g,h]=1$. Hence $q^*\pi_6(S^3)$ is central. For $g=x$ and $h=y$ the composite $(x\wedge y)\circ\bar\Delta$ is the quotient map $M\to S^6$, so $[x,y]=q^*\langle\iota_3,\iota_3\rangle$. By Bott \cite{bott}, the Samelson product $\langle\iota_3,\iota_3\rangle$ generates $\pi_6(S^3)\cong\mathbb{Z}/12$. Thus $c$ generates $\ker i^*$, which proves the first two assertions. The presentation follows, since the group defined by it is a central extension of $\mathbb{Z}^2$ by $\mathbb{Z}/12$ with the same generators and the same commutator.

(b) Because $S^3$ is a group, precomposition $f^*\colon N\to N$ is a homomorphism, and $(f\circ g)^*=g^*\circ f^*$. A based map $f\colon M\to M$ is determined up to based homotopy by $(f_1,f_2)=(\mathrm{pr}_1f,\mathrm{pr}_2f)=(f^*x,f^*y)$, so $f\mapsto f^*$ is injective. If $f$ is a self-equivalence, $f^*$ is an automorphism, and $f^*$ induces on $N/\langle c\rangle\cong H^3(M;\mathbb{Z})$ the map $f^*$ in cohomology, whose determinant is the degree of $f$, because $H^6(M)=H^3(M)\wedge H^3(M)$. So orientation-preserving equivalences map into $\Aut^+(N)$. Conversely, given $\psi\in\Aut^+(N)$ put $f=(\psi(x),\psi(y))\colon M\to S^3\times S^3$. Then $f^*$ and $\psi$ are homomorphisms that agree on the generators $x,y$, so $f^*=\psi$. The map $f$ induces an isomorphism on $H^3$ and, having degree $\det=1$, on $H^6$, so it is a homology isomorphism of simply connected spaces and hence a homotopy equivalence.
\end{proof}
\begin{theorem}\label{thm:S3xS3_split}
There is a split short exact sequence
\[
  0 \longrightarrow \mathbb{Z}/12 \oplus \mathbb{Z}/12 \longrightarrow \Aut_\bullet(S^3 \times S^3) \xrightarrow{\pi} SL_2(\mathbb{Z}) \longrightarrow 0 .
\]
Identify the kernel with $V=\Hom(\mathbb{Z}^2,\mathbb{Z}/12)$. Then $SL_2(\mathbb{Z})$ acts on $V$ by $A\cdot\lambda=\lambda\circ A^{-1}$, the dual of the standard action on $(\mathbb{Z}/12)^2$. Here $A$ is the matrix induced on $N/\langle c\rangle$ by an automorphism of $N$; since $f\mapsto f^*$ reverses composition, the action of $\pi(f)$ in $\Aut_\bullet(M)$ differs from this one by the inverse transpose, which is an inner automorphism of $SL_2(\mathbb{Z})$, so the semidirect product is the same. Hence
\[
  \Aut_\bullet(S^3 \times S^3) \cong (\mathbb{Z}/12)^2\rtimes SL_2(\mathbb{Z}).
\]
This group is not a direct product $(\mathbb{Z}/12)^2\times SL_2(\mathbb{Z})$: its abelianization is $\mathbb{Z}/12$, while that of the direct product is $(\mathbb{Z}/12)^3$.
\end{theorem}
\begin{proof}
By Lemma~\ref{lem:AutN} we may work in $\Aut^+(N)$. Taking the induced matrix on $N/\langle c\rangle=\mathbb{Z}^2$ gives a homomorphism $\Aut^+(N)\to SL_2(\mathbb{Z})$. It is onto: for $A=\left(\begin{smallmatrix}a&b\\c'&d\end{smallmatrix}\right)$ with $ad-bc'=1$, the assignment $x\mapsto x^ay^{c'}$, $y\mapsto x^by^d$, $c\mapsto c$ respects the defining relations, because $N$ has class $2$ and so $[x^ay^{c'},x^by^d]=c^{ad-bc'}=c$, and it induces $A$ on $N/\langle c\rangle$ and the identity on $\langle c\rangle$, hence is bijective.

An automorphism in the kernel has the form $\tau_\lambda(g)=g\,c^{\lambda(\bar g)}$ for $\lambda\in V$, where $\bar g$ is the image of $g$ in $\mathbb{Z}^2$ (these are the pinch-fold twists considered below). This identifies the kernel with $V\cong(\mathbb{Z}/12)^2$. The kernel is thus a copy of $\pi_6(S^3)^2$, where $\pi_6(S^3)\cong\mathbb{Z}/12$ by Toda \cite{toda}; it is the kernel in Sieradski's exact sequence \cite[Example~3]{sieradski}. That the commutator $c=[x,y]$ has order $12$ is Bott's computation of the Samelson product \cite{bott}, used in Lemma~\ref{lem:AutN}. If $\varphi\in\Aut^+(N)$ induces $A$, then $\varphi(c)=c$, and a direct computation gives $\varphi\tau_\lambda\varphi^{-1}=\tau_{\lambda\circ A^{-1}}$. This is the stated action.

For the splitting we use the presentation $SL_2(\mathbb{Z})=\langle S,U\mid S^4=1,\ S^2=U^3\rangle$ with
\[
S=\begin{pmatrix}0&-1\\1&0\end{pmatrix},\qquad U=\begin{pmatrix}0&-1\\1&1\end{pmatrix}
\]
(the amalgam $\mathbb{Z}/4*_{\mathbb{Z}/2}\mathbb{Z}/6$, with $S^2=U^3=-I$). Define $\tilde S,\tilde U\in\Aut^+(N)$ by
\[
\tilde S\colon x\mapsto y,\ y\mapsto x^{-1}c,\ c\mapsto c;\qquad \tilde U\colon x\mapsto y,\ y\mapsto x^{-1}y\,c,\ c\mapsto c .
\]
They induce $S$ and $U$ on $N/\langle c\rangle$ (we write the matrix of an automorphism with the images of $x$ and $y$ as columns). In the class-$2$ group $N$ the commutator is bilinear and $c$ is central, so one computes directly:
\begin{align*}
\tilde S^2(x)&=x^{-1}c,&\tilde S^2(y)&=y^{-1}c,&\tilde S^4(x)&=(x^{-1}c)^{-1}c=x,&\tilde S^4(y)&=y,\\
\tilde U^2(x)&=x^{-1}yc,&\tilde U^2(y)&=x^{-1}c,&\tilde U^3(x)&=x^{-1}c,&\tilde U^3(y)&=y^{-1}c.
\end{align*}
Hence $\tilde S^4=1$ and $\tilde U^3=\tilde S^2$. (For the lifts without the factor $c$ the relation $\tilde U^3=\tilde S^2$ fails, so the central correction matters.) By the presentation, $S\mapsto\tilde S$, $U\mapsto\tilde U$ defines a homomorphism $SL_2(\mathbb{Z})\to\Aut^+(N)$ that is a section of the projection, so the sequence splits.

The action is not trivial: $S$ sends $\tau_\lambda$ to $\tau_{\lambda\circ S^{-1}}$, which differs from $\tau_\lambda$ for $\lambda=(1,0)$. For $G=V\rtimes SL_2(\mathbb{Z})$ we have $G^{ab}=V_{SL_2(\mathbb{Z})}\times SL_2(\mathbb{Z})^{ab}$, where $V_{SL_2(\mathbb{Z})}$ denotes the coinvariants. The map $\lambda\mapsto\lambda\circ U^{-1}-\lambda$ on $V$ has determinant $\det(U^{-1}-I)=1$ and is therefore invertible on $(\mathbb{Z}/12)^2$, so the coinvariants vanish. Since $SL_2(\mathbb{Z})^{ab}\cong\mathbb{Z}/12$, we get $G^{ab}\cong\mathbb{Z}/12$.
\end{proof}

\begin{corollary}\label{cor:S3S3_braid}
The group $\widehat\Omega^{\mathrm{spin}}_7(P,M)$ has order $144$, and $\gamma$ maps it bijectively onto $\ker\pi\cong(\mathbb{Z}/12)^2$.
\end{corollary}
\begin{proof}
By Corollary~\ref{im}, $\gamma$ maps onto $\ker\pi$, which has $144$ elements by Theorem~\ref{thm:S3xS3_split}. By Theorem~\ref{thm:bordism_S3S3}(ii), the domain has at most $144$ elements. Hence $\gamma$ is a bijection.
\end{proof}

Thus the upper bound of Theorem~\ref{thm:bordism_S3S3}(ii) is attained. In particular $E^\infty_{5,2}\cong(\mathbb{Z}/2)^2$, so the differential $d_3\colon E^3_{8,0}\to E^3_{5,2}$ vanishes.

\begin{remark}
Sieradski \cite[Example~3]{sieradski} obtains an exact sequence $0\to\mathbb{Z}/12\oplus\mathbb{Z}/12\to G(S^3\times S^3)\to GL_2(\mathbb{Z})\to1$, where $G$ also contains the equivalences of degree $-1$. Its orientation-preserving part is the sequence of Theorem~\ref{thm:S3xS3_split}. The splitting and the action of $SL_2(\mathbb{Z})$ on the kernel are not part of the exact sequence as stated by Sieradski.
\end{remark}
\subsection{Connected sums of copies of $S^3\times S^3$}
\label{subsec:connsum}

For $r\ge1$ let $M_r=\connsum_r(S^3\times S^3)$, with basepoint in the connecting region. Let $s_{i,1},s_{i,2}\colon S^3\to M_r$ be the inclusions of the two spheres of the $i$-th summand, so that $M_r=(\bigvee_{2r}S^3)\cup_w e^6$ with $w=\sum_i[s_{i,1},s_{i,2}]$, and let $q\colon M_r\to S^6$ be the pinch map onto the top cell. Put $N_r:=[M_r,S^3]_\bullet$, a group under pointwise multiplication in $S^3$, and let $x_i,y_i\in N_r$ be the classes obtained by collapsing all summands except the $i$-th and projecting onto the first, respectively second, $S^3$-factor. Thus $x_i s_{k,1}$ is the identity of $S^3$ for $k=i$ and null otherwise, $x_is_{k,2}$ is null, and similarly for $y_i$.

For $\kappa\in\pi_6(S^3)$ put $t_{i,j}(\kappa):=t_{\tilde\kappa}$ with $\tilde\kappa=s_{i,j}\circ\kappa\in\pi_6(M_r)$, and let $f_{i,j}:=t_{i,j}(\varpi)$, where $\varpi$ is a generator of $\pi_6(S^3)\cong\mathbb{Z}/12$.

\begin{theorem}\label{non-splitting_extension}
Let $r\ge2$.
\begin{enumerate}
\item[(a)] The elements $f_{i,j}$ commute and generate a subgroup $K_r^0\cong(\mathbb{Z}/12)^{2r}$ of the kernel of $\pi\colon\Aut_\bullet(M_r)\to\Aut_\bullet(P)$.
\item[(b)] Let $\tau_{ik}$ be an orientation-preserving diffeomorphism of $M_r$ that exchanges the $i$-th and $k$-th summands, in the sense that $\tau_{ik}\circ s_{i,j}\simeq s_{k,j}$, $\tau_{ik}\circ s_{k,j}\simeq s_{i,j}$ and $\tau_{ik}\circ s_{l,j}\simeq s_{l,j}$ for $l\ne i,k$ and $j=1,2$. Then $\tau_{ik}f_{i,j}\tau_{ik}^{-1}=f_{k,j}$ and $\tau_{ik}f_{k,j}\tau_{ik}^{-1}=f_{i,j}$ for $j=1,2$, and $\tau_{ik}f_{l,j}\tau_{ik}^{-1}=f_{l,j}$ for $l\ne i,k$. In particular the diffeomorphisms $\tau_{ik}$ normalize $K_r^0$, they generate a group that acts on $K_r^0$ through $\Sigma_r$ by permuting the $r$ pairs of generators, and so the conjugation action of $\Aut_\bullet(M_r)$ on $K_r^0$ is not trivial.
\end{enumerate}
\end{theorem}
\begin{proof}
(a) Each $\tilde\kappa=s_{i,j}\kappa$ lies in the $3$-skeleton $\bigvee_{2r}S^3$, on which every pinch-fold map $t_{\tilde\kappa'}$ with $\tilde\kappa'$ in the $3$-skeleton restricts to the inclusion. Lemma~\ref{lem:coaction} therefore shows that all the maps $t_{i,j}(\kappa)$ commute and that $t_{i,j}(\kappa)\circ t_{i,j}(\kappa')\simeq t_{i,j}(\kappa+\kappa')$. Hence $a\mapsto F_a=\prod_{i,j}t_{i,j}(a_{i,j}\varpi)$ is a homomorphism $(\mathbb{Z}/12)^{2r}\to\Aut_\bullet(M_r)$ with image $K_r^0$. Each $t_{i,j}(\kappa)$ is the identity on the $3$-skeleton, so it induces the identity on $\pi_3(M_r)$ and on $P$, and $K_r^0\subset\ker\pi$.

For injectivity we compute the action on $N_r$. Let $t=t_{i,2}(\kappa)$. Lemma~\ref{lem:coaction} with $G=S^3$ gives $\varphi\circ t\simeq\varphi\cdot q^*(\varphi\circ s_{i,2}\circ\kappa)$ for every $\varphi\in N_r$. For $\varphi=y_i$ we have $y_is_{i,2}=\id$, so $y_i\circ t=y_i\cdot q^*\kappa$. For the other generators, $x_k s_{i,2}$ and $y_ks_{i,2}$ ($k\ne i$) are null, so $x_k\circ t=x_k$ and $y_k\circ t=y_k$, and $x_i\circ t=x_i$. Also $q\circ s_{i,2}\circ\kappa\simeq*$, so $q\circ t\simeq q$ and $t$ fixes every class $q^*\beta$. The twists $t_{i,1}$ act in the same way on $x_i$. Since precomposition with a map is a homomorphism of $N_r$, composing gives
\[
  y_i\circ F_a=y_i\cdot q^*(a_{i,2}\varpi),\qquad x_i\circ F_a=x_i\cdot q^*(a_{i,1}\varpi).
\]
The cofibre sequence $S^5\xrightarrow{w}\bigvee_{2r}S^3\to M_r\xrightarrow{q}S^6\xrightarrow{\Sigma w}\bigvee_{2r}S^4$ gives an exact sequence of groups
\[
  \pi_4(S^3)^{2r}\xrightarrow{(\Sigma w)^*}\pi_6(S^3)\xrightarrow{q^*}N_r .
\]
Since $w$ is a sum of Whitehead products, $\Sigma w\simeq0$, so $q^*$ is injective. If $F_a\simeq\id$, then $q^*(a_{i,j}\varpi)=1$ for all $i,j$, hence $a_{i,j}\varpi=0$ and $a=0$.

(b) For an orientation-preserving diffeomorphism $h$ one has $h\circ t_{\tilde\kappa}\circ h^{-1}\simeq t_{h\circ\tilde\kappa}$ by Lemma~\ref{lem:coaction}(c). The hypotheses on $\tau_{ik}$ give $\tau_{ik}\circ s_{i,j}\kappa\simeq s_{k,j}\kappa$, $\tau_{ik}\circ s_{k,j}\kappa\simeq s_{i,j}\kappa$ and $\tau_{ik}\circ s_{l,j}\kappa\simeq s_{l,j}\kappa$, and the claim follows.
\end{proof}

\begin{theorem}\label{thm:Mr_kernel}
Let $r\ge1$ and $M_r=\connsum_r(S^3\times S^3)$.
\begin{enumerate}
\item[(a)] The image of $\pi\colon\Aut_\bullet(M_r)\to\Aut_\bullet(P)$ is $\Isom[\pi_3(M_r),s_{M_r}]\cong Sp_{2r}(\mathbb{Z})$.
\item[(b)] The group $\widehat\Omega^{\mathrm{spin}}_7(P,M_r)$ has order $144^r$, $\gamma$ maps it bijectively onto $\ker\pi$, and $\ker\pi=K_r^0\cong(\mathbb{Z}/12)^{2r}$.
\end{enumerate}
Hence there is an exact sequence $1\to(\mathbb{Z}/12)^{2r}\to\Aut_\bullet(M_r)\to Sp_{2r}(\mathbb{Z})\to1$.
\end{theorem}
\begin{proof}
(a) Every self-equivalence of degree one preserves $s_{M_r}$, so the image lies in the isometry group. Conversely, let $F\in Sp_{2r}(\mathbb{Z})$, and write $s_1,\dots,s_{2r}$ for the inclusions of the $3$-spheres, ordered so that $w=\sum_{a<b}J_{ab}[s_a,s_b]$ with $J$ the standard skew form. Define $f_0\colon\bigvee_{2r}S^3\to\bigvee_{2r}S^3$ by $f_0s_a=\sum_bF_{ba}s_b$. The Whitehead product is natural, bilinear and, on $\pi_3$, skew: $[s_b,s_a]=-[s_a,s_b]$; moreover $[s_a,s_a]=0$ since $S^3$ is an $H$-space. Let $J^+_{ab}=J_{ab}$ for $a<b$ and $J^+_{ab}=0$ otherwise, so that $J=J^+-(J^+)^T$, and put $G=FJ^+F^T$. Then
\[
f_0\circ w=\sum_{c,d}G_{cd}[s_c,s_d]=\sum_{c<d}(G_{cd}-G_{dc})[s_c,s_d]=\sum_{c<d}(FJF^T)_{cd}[s_c,s_d]=w,
\]
because $FJF^T=J$ (the group $Sp_{2r}(\mathbb{Z})$ is closed under transposition). Hence $f_0$ extends to a map $f\colon M_r\to M_r$. It induces $F$ on $\pi_3$, and since $F$ preserves the cup product pairing $H^3\times H^3\to H^6=\mathbb{Z}$ the degree of $f$ is $1$. So $f$ is a homology isomorphism of simply connected spaces, hence a homotopy equivalence, and $\pi(f)=F$.

(b) By Corollary~\ref{im}, $\gamma$ maps $\widehat\Omega^{\mathrm{spin}}_7(P,M_r)$ onto $\ker\pi$, and $K_r^0\subseteq\ker\pi$ has order $144^r$ by Theorem~\ref{non-splitting_extension}(a). So it suffices to show that $|\widehat\Omega^{\mathrm{spin}}_7(P,M_r)|\le144^r$. Here $P=K(\mathbb{Z}^{2r},3)$. The proof of Theorem~\ref{thm:bordism_S3S3} applies with the following changes. The stable splitting of $M_r$ gives $\widehat\Omega^{\mathrm{spin}}_6(M_r)=0$, so $\widehat\Omega^{\mathrm{spin}}_7(P,M_r)=\operatorname{coker}[c_*\colon\Omega^{\mathrm{spin}}_7(M_r)\to\Omega^{\mathrm{spin}}_7(P)]$. By the K\"unneth formula and Lemma~\ref{lem:homology_K(ZZ,3)} applied to each factor of $P=K(\mathbb{Z},3)^{2r}$, $H_7(P;\mathbb{Z})\cong(\mathbb{Z}/3)^{2r}$ and $H_4(P;\mathbb{Z})=0$. The mod $2$ cohomology of $P$ is polynomial on the classes $\iota_a$ and $\mathrm{Sq}^2\iota_a$, $1\le a\le2r$, in degrees at most $8$. So $H^5(P;\mathbb{Z}/2)\cong(\mathbb{Z}/2)^{2r}$, $H^7(P;\mathbb{Z}/2)=0$, and $H^6(P;\mathbb{Z}/2)$ has the basis $\iota_a^2$ and $\iota_a\iota_b$ ($a<b$), of dimension $2r+\binom{2r}{2}$. Thus $E^2_{5,2}\cong(\mathbb{Z}/2)^{2r}$, $E^2_{6,1}\cong(\mathbb{Z}/2)^{2r+\binom{2r}{2}}$, $E^2_{7,0}\cong(\mathbb{Z}/3)^{2r}$ and $E^2_{4,3}=0$. Since $H_7(P;\mathbb{Z})$ has no $2$-torsion, reduction mod $2$ maps $H_8(P;\mathbb{Z})$ onto $H_8(P;\mathbb{Z}/2)$, so $d_2\colon E^2_{8,0}\to E^2_{6,1}$ has the rank of $\mathrm{Sq}^2\colon H^6(P;\mathbb{Z}/2)\to H^8(P;\mathbb{Z}/2)$. We have $\mathrm{Sq}^2(\iota_a^2)=(\mathrm{Sq}^1\iota_a)^2=0$ and $\mathrm{Sq}^2(\iota_a\iota_b)=\mathrm{Sq}^2\iota_a\cdot\iota_b+\iota_a\cdot\mathrm{Sq}^2\iota_b$, and these $\binom{2r}{2}$ elements involve disjoint sets of monomials in the basis of $H^8(P;\mathbb{Z}/2)$, so the rank is $\binom{2r}{2}$ and $E^3_{6,1}$ has order $2^{2r}$. As in the case $r=1$, $F_3=F_4$ lies in the image of $c_*$, because $P^{(3)}=\bigvee_{2r}S^3\subset M_r$, and the cokernel of $c_*$ has order at most $|E^2_{5,2}|\cdot|E^3_{6,1}|\cdot|E^2_{7,0}|=2^{2r}\cdot2^{2r}\cdot3^{2r}=144^r$.
\end{proof}

For $r=1$ this recovers Theorem~\ref{thm:S3xS3_split} and Corollary~\ref{cor:S3S3_braid}. For $r\ge2$ the kernel is abstractly isomorphic to $H^3(M_r;\pi_6(S^3))$, and by Theorem~\ref{non-splitting_extension}(b) the permutation diffeomorphisms act on it non-trivially. The classes $[\iota_a,\iota_b]\circ\eta_5$ in $\pi_6(M_r)$ give no new elements of the kernel. Indeed, each of them is a Whitehead product $[s_d,s_c\circ\eta]$ of a generator with an element of $\pi_4(M_r)$, and such products lie in the subgroup of $\pi_6(M_r)$ that gives trivial twists.

The braid does not decide whether the extension splits. This question is answered by Crowley \cite{crowley}, who describes the extensions $0\to H_p(M_r)\otimes\pi_{2p}(S^p)\to\mathcal{E}(M_r)\to Sp_{2r}(\mathbb{Z})\to1$ for $M_r=\connsum_r(S^p\times S^p)$ with $p=3,7$ and shows that they split if and only if $r=1$. In particular the extension of Theorem~\ref{thm:Mr_kernel} does not split for $r\ge2$. Crowley's identification of the kernel and the image agrees with Theorem~\ref{thm:Mr_kernel}; here both are obtained from the braid.

\subsection{The case $M=S^3\times S^4$}
\label{subsec:S3xS4}

Let $M=S^3\times S^4$. Then $m=7=2n-1$ with $n=4$, so $\B=BO\langle5\rangle=BString$. The manifold is parallelizable, hence string, and simply connected, so (T) holds and Theorem~\ref{thm:braid_exists} applies. Here $\pi_3(M)=\mathbb{Z}$, $\pi_4(M)=\pi_4(S^3)\oplus\pi_4(S^4)\cong\mathbb{Z}/2\oplus\mathbb{Z}$, and $P$ is a two-stage Postnikov system $K(\mathbb{Z}/2\oplus\mathbb{Z},4)\to P\to K(\mathbb{Z},3)$ with $k$-invariant $(\mathrm{Sq}^2\iota_3,0)$. An equivalence of $P$ that induces the identity on $\pi_3$ and $\pi_4$ is homotopic to the identity: since $H^4(K(\mathbb{Z},3);\pi_4)=0$, the restriction $H^4(P;\pi_4)\to\mathrm{Hom}(\pi_4,\pi_4)$ is injective, and it records the effect on $\pi_4$. So $\pi(f)$ is determined by the action of $f$ on $\pi_3(M)$ and $\pi_4(M)$.

We use the following facts from Toda \cite[Chapter~V]{toda}. The group $\pi_7(S^3)\cong\mathbb{Z}/2$ is generated by $\nu'\eta_6$, and $\pi_7(S^4)\cong\mathbb{Z}\{\nu_4\}\oplus E\pi_6(S^3)$ with $E\pi_6(S^3)\cong\mathbb{Z}/12$. The Whitehead square is
\[
[\iota_4,\iota_4]=\pm\bigl(2\nu_4-E\nu'\bigr),
\]
where $\nu'$ generates the $2$-primary part $\mathbb{Z}/4$ of $\pi_6(S^3)$, and $\eta_3\circ\nu_4=\nu'\circ\eta_6$, where $\eta_3\in\pi_4(S^3)$ is the Hopf class. Let $D=\langle[\iota_4,\iota_4]\rangle\subset\pi_7(S^4)$ and $Q=\pi_7(S^4)/D$. Since $[\iota_4,\iota_4]=2\nu_4-u\,\vartheta$ with $\vartheta$ a generator of $E\pi_6(S^3)\cong\mathbb{Z}/12$ and $u$ odd, the group $Q$ is generated by the classes of $\nu_4$ and $\vartheta$ subject to $12\vartheta=0$ and $2\nu_4=u\vartheta$, so its $2$-primary part is $\mathbb{Z}/8$, generated by $\nu_4$, and its $3$-primary part is $\mathbb{Z}/3$. Thus $Q\cong\mathbb{Z}/24$. (This is also the image of $\pi_7(S^4)$ in the stable group $\pi_3^s$, as it must be, since $D$ is the kernel of the suspension.) Left composition with $\eta_3$ is a homomorphism $\pi_7(S^4)\to\pi_7(S^3)$, and it kills $D$, because $\eta_3\circ[\iota_4,\iota_4]=[\eta_3,\eta_3]=0$, a Whitehead product in the $H$-space $S^3$. So it induces a homomorphism $\varepsilon\colon Q\to\mathbb{Z}/2$ with $\varepsilon(\nu_4)=\nu'\eta_6\ne0$, the unique surjection.

\begin{theorem}\label{thm:S3S4}
For $M=S^3\times S^4$:
\begin{enumerate}
\item[(a)] The kernel of $\pi\colon\Aut_\bullet(M)\to\Aut_\bullet(P)$ is the group of pinch-fold twists, $\ker\pi\cong\pi_7(S^3)\oplus Q\cong\mathbb{Z}/2\oplus\mathbb{Z}/24$.
\item[(b)] The image of $\pi$ is a group $G\cong(\mathbb{Z}/2)^2$, generated by the images of the diffeomorphisms $s(x,y)=(\bar x,r(y))$ and $\sigma(x,y)=(x\cdot\eta_3(y),y)$, where $x\mapsto\bar x$ is quaternionic conjugation on $S^3$ and $r$ is a reflection of $S^4$.
\item[(c)] The extension splits, and
\[
\Aut_\bullet(S^3\times S^4)\cong(\mathbb{Z}/2\oplus\mathbb{Z}/24)\rtimes(\mathbb{Z}/2)^2,
\]
where the two generators act by $(a,b)\mapsto(a,-b)$ and $(a,b)\mapsto(a+\varepsilon(b),b)$ respectively. In particular $\Aut_\bullet(S^3\times S^4)$ is finite, of order $192$.
\end{enumerate}
\end{theorem}
\begin{proof}
Write $M=(S^3\vee S^4)\cup_we^7$ with $w=[\iota_3,\iota_4]$, let $j_1,j_2$ be the inclusions of the two spheres, $\mathrm{pr}_1,\mathrm{pr}_2$ the projections, and $q\colon M\to S^7$ the pinch map. A self-map $f$ of $M$ is the pair $(f_1,f_2)=(\mathrm{pr}_1f,\mathrm{pr}_2f)\in[M,S^3]_\bullet\times[M,S^4]_\bullet$.

\emph{Restrictions.} For $f\in\Aut_\bullet(M)$ we have $f_1j_1=e_3\iota_3$, $f_2j_1=0$ (as $\pi_3(S^4)=0$), $f_2j_2=e_4\iota_4$ and $f_1j_2=a\eta_3$ with $e_3,e_4\in\mathbb{Z}$ and $a\in\mathbb{Z}/2$. Since $f$ is a homotopy equivalence, $e_3,e_4\in\{\pm1\}$, and $\deg f=e_3e_4$, so $e_3=e_4=:e$ for $f\in\Aut_\bullet(M)$. The action of $f$ on $\pi_3(M)=\mathbb{Z}$ is multiplication by $e$, and on $\pi_4(M)=\mathbb{Z}/2\oplus\mathbb{Z}$ it sends $(\eta_3,0)\mapsto(\eta_3,0)$ and $(0,\iota_4)\mapsto(a\eta_3,e\iota_4)$. So $\pi(f)$ is determined by $(e,a)$, and $f\in\ker\pi$ if and only if $e=1$ and $a=0$, that is $f|_{S^3\vee S^4}\simeq j_1\vee j_2$.

\emph{The kernel.} Let $f\in\ker\pi$. Then $f_1$ restricts to $(\iota_3,0)$ and $f_2$ to $(0,\iota_4)$ on $S^3\vee S^4$, and these are the restrictions of $\mathrm{pr}_1$ and $\mathrm{pr}_2$. By Lemma~\ref{lem:orbit} (with $Y=S^3$, $g=\iota_3$ in the interchanged form, and $Y=S^4$, $g=\iota_4$) there are $b_1\in\pi_7(S^3)$ and $b_2\in\pi_7(S^4)$ with $f_1=b_1\cdot\mathrm{pr}_1$ and $f_2=b_2\cdot\mathrm{pr}_2$ for the coaction. Put $\tilde\kappa=j_1b_1+j_2b_2\in\pi_7(M)$. Since $\mathrm{pr}_i\circ t_{\tilde\kappa}=(\mathrm{pr}_i\tilde\kappa)\cdot\mathrm{pr}_i$ and $\mathrm{pr}_1\tilde\kappa=b_1$, $\mathrm{pr}_2\tilde\kappa=b_2$, we get $f=t_{\tilde\kappa}$. Conversely every $t_{\tilde\kappa}$ with $\tilde\kappa\in\pi_7(M)=\pi_7(S^3)\oplus\pi_7(S^4)$ lies in $\ker\pi$: it is the identity on $S^3\vee S^4$ and has degree one, because the Hurewicz image of $\tilde\kappa$ in $H_7(M)$ vanishes.

By Lemma~\ref{lem:coaction}(b) (take $K=S^3\vee S^4$) the map $\Psi\colon\pi_7(M)\to\ker\pi$, $\tilde\kappa\mapsto t_{\tilde\kappa}$, is a surjective homomorphism. Its kernel consists of the classes $\tilde\kappa=(b_1,b_2)$ with $b_1\cdot\mathrm{pr}_1\simeq\mathrm{pr}_1$ and $b_2\cdot\mathrm{pr}_2\simeq\mathrm{pr}_2$. By Lemma~\ref{lem:orbit} the stabilizer of $\mathrm{pr}_2$ is $[\pi_4(S^4),\iota_4]=D$, and the stabilizer of $\mathrm{pr}_1$ is $[\pi_5(S^3),\iota_3]=0$, since Whitehead products vanish in the $H$-space $S^3$. Hence $\ker\Psi=0\oplus D$ and $\ker\pi\cong\pi_7(S^3)\oplus Q\cong\mathbb{Z}/2\oplus\mathbb{Z}/24$. This proves (a).

\emph{The image.} The maps $s$ and $\sigma$ are orientation-preserving based diffeomorphisms; $s$ has degree $(-1)(-1)=1$ because both $x\mapsto\bar x$ on $S^3$ and $r$ on $S^4$ have degree $-1$, and $\sigma$ has the inverse $(x,y)\mapsto(x\eta_3(y)^{-1},y)$. They realize $(e,a)=(-1,0)$ and $(1,1)$. The assignment $f\mapsto(e,a)$ is a homomorphism $\Aut_\bullet(M)\to\{\pm1\}\times\mathbb{Z}/2$: if $g$ has invariants $(e',a')$, then $g\circ j_2=j_1(a'\eta_3)+j_2(e'\iota_4)$, so $(f\circ g)_1j_2=f_1j_1(a'\eta_3)+f_1j_2(e'\iota_4)=a'\eta_3+e'a\eta_3$, which is $(a+a')\eta_3$ in $\pi_4(S^3)=\mathbb{Z}/2$, and the signs multiply. Hence $G=\operatorname{im}\pi$ is the subgroup of $\{\pm1\}\times\mathbb{Z}/2$ generated by the invariants of $s$ and $\sigma$, which are $(-1,0)$ and $(1,1)$, and all four pairs occur, so $G\cong(\mathbb{Z}/2)^2$. This proves (b), and $|\Aut_\bullet(M)|=48\cdot4=192$.

\emph{The action and the splitting.} Write elements of $K:=\ker\pi$ as $(a,b)$, with $a\in\pi_7(S^3)=\mathbb{Z}/2$ and $b\in Q$, meaning $t_{\tilde\kappa}$ for $\tilde\kappa=j_1a+j_2b$. By Lemma~\ref{lem:coaction}(c), for a diffeomorphism $h$ of degree one, $h\,t_{\tilde\kappa}\,h^{-1}=t_{h_*\tilde\kappa}$. For $s$: $s\circ j_1a=j_1(\bar{\ }\circ a)$ and $\bar{\ }\circ a=-a=a$, since inversion in $S^3$ acts on homotopy groups by $-1$ and $\pi_7(S^3)=\mathbb{Z}/2$; and $s\circ j_2b=j_2(r\circ b)$ with $r\circ b=(-\iota_4)\circ b=-b+[\iota_4,\iota_4]\circ H(b)\equiv-b$ modulo $D$. So $s_*(a,b)=(a,-b)$. For $\sigma$: $\sigma\circ j_1a=j_1a$ and $\sigma\circ j_2b=j_1(\eta_3\circ b)+j_2b$, because $\sigma(*,y)=(\eta_3(y),y)$. So $\sigma_*(a,b)=(a+\varepsilon(b),b)$.

Now $s^2=\id$, and $\sigma^2(x,y)=(x\eta_3(y)^2,y)\simeq\id$, since $\eta_3\cdot\eta_3=2\eta_3=0$ in $\pi_4(S^3)$. The commutator $c_0=s\sigma s^{-1}\sigma^{-1}$ lies in $\ker\pi$, because $G$ is abelian, and $\mathrm{pr}_2\circ c_0=\mathrm{pr}_2$ strictly (the second coordinate is reflected twice and shifted trivially), so the $Q$-component of $c_0$ is zero by Lemma~\ref{lem:orbit}. Thus $c_0=(a_0,0)$. Put $k=\id$ if $a_0=0$ and $k=(0,\nu_4)\in K$ if $a_0\ne0$, and $s'=sk$. Then $s'^2=s\,k\,s\,k=(s_*k)\,k=(0,-\nu_4)(0,\nu_4)=1$ or $s'^2=1$ trivially. Moreover
\[
[s',\sigma]=s\,(k\sigma k^{-1}\sigma^{-1})\,s^{-1}\cdot c_0=s_*\bigl((0,\nu_4)-\sigma_*(0,\nu_4)\bigr)+c_0=(\varepsilon(\nu_4),0)+(a_0,0)=0,
\]
using $\varepsilon(\nu_4)\ne0$ and $|\mathbb{Z}/2|=2$. So $s'$ and $\sigma$ are commuting involutions, and $\langle s',\sigma\rangle\cong(\mathbb{Z}/2)^2$ maps isomorphically onto $G$. This is a splitting, and the action of $G$ on $K$ is the one stated. This proves (c).
\end{proof}

\begin{remark}
Since $\gamma$ maps $\widehat\Omega^{BString}_8(P,M)$ onto $\ker\pi$ (Corollary~\ref{im}), the relative string bordism group has at least $48$ elements.
\end{remark}

\begin{remark}
The order $192$ can be compared with Pave\v{s}i\'c's description of $\mathcal{E}(S^m\times S^n)$ for $m<n$ \cite[Proposition~5.1]{pavesic}, namely $\mathcal{E}(S^m\times S^n)=(\pi_n(\mathrm{aut}_1S^m)\rtimes\mathbb{Z}/2)\cdot(\pi_m(\mathrm{aut}_1S^n)\rtimes\mathbb{Z}/2)$. Here $\mathcal{E}(X)$ is the group of homotopy classes of self-equivalences of $X$ and $\mathrm{aut}_1X$ is the component of the identity in the space of self-maps of $X$. The formula applies when $m\ne n$, which is the case in which self-equivalences of $S^m\times S^n$ can be diagonalized in Pave\v{s}i\'c's sense \cite[Corollary~2.2]{pavesic}, namely when the components $S^m\to S^m$ and $S^n\to S^n$ of every self-equivalence are again self-equivalences. The groups $\pi_k(\mathrm{aut}_1S^m)$ are computed by the exact sequence of Koh
\[
\pi_{k+1}(S^m)\xrightarrow{[\iota_m,-]}\pi_{k+m}(S^m)\to\pi_k(\mathrm{aut}_1S^m)\to\pi_k(S^m)\xrightarrow{[\iota_m,-]}\pi_{k+m-1}(S^m).
\]
Since $S^3$ is a group, $\mathrm{aut}_1S^3\simeq S^3\times\Omega^3_0S^3$, so $\pi_4(\mathrm{aut}_1S^3)\cong\pi_4(S^3)\oplus\pi_7(S^3)$ has order $4$. Koh's sequence gives $\pi_3(\mathrm{aut}_1S^4)\cong\pi_7(S^4)/[\iota_4,\iota_4]\cong\mathbb{Z}/24$. Hence $|\mathcal{E}(S^3\times S^4)|=(4\cdot2)(24\cdot2)=384$, and the subgroup of orientation-preserving classes has order $192$, in agreement with Theorem~\ref{thm:S3S4}.
\end{remark}

\subsection{The case $M=S^4\times S^4$}
\label{subsec:S4xS4}

Let $M=S^4\times S^4$. It is parallelizable, $m=8=2n$ with $n=4$, so $\B=BO\langle5\rangle=BString$. Since $\pi_4(M)\cong\mathbb{Z}^2$ and $\pi_i(M)=0$ for $i<4$, we have $P=K(\mathbb{Z}^2,4)$ and $\Aut_\bullet(P)=GL_2(\mathbb{Z})$. Hypothesis (T) holds because $M$ is simply connected, so Theorem~\ref{thm:braid_exists} applies.

In the basis of $H^4(M;\mathbb{Z})$ dual to the two factors the intersection form is symmetric,
\[
s_M=\begin{pmatrix}0&1\\1&0\end{pmatrix},
\]
since the cup product on $H^4$ is graded commutative with sign $(-1)^{4\cdot4}=+1$. Writing $A=\left(\begin{smallmatrix}a&b\\c&d\end{smallmatrix}\right)$, the condition $A^Ts_MA=s_M$ becomes $2ac=0$, $2bd=0$, $ad+bc=1$. If $a\ne0$ then $c=0$, so $ad=1$ and $a=d=\pm1$, and then $b=0$. If $a=0$ then $bc=1$, so $b=c=\pm1$ and $d=0$. Hence
\[
\Isom[\mathbb{Z}^2,s_M]=\{\pm I,\pm\sigma_0\}\cong(\mathbb{Z}/2)^2,\qquad\sigma_0=\begin{pmatrix}0&1\\1&0\end{pmatrix}.
\]
In contrast with $S^3\times S^3$, no isometry is excluded by orientation: the swap of the two factors has degree $(-1)^{4\cdot4}=+1$.

From Toda \cite[Chapter~V]{toda} we use $\pi_5(S^4)=\mathbb{Z}/2\{\eta_4\}$ and $\pi_8(S^4)=\mathbb{Z}/2\{\nu_4\eta_7\}\oplus\mathbb{Z}/2\{E\nu'\eta_7\}$, together with $[\iota_4,\iota_4]=\pm(2\nu_4-E\nu')$. Since $\eta_4=E\eta_3$ is a suspension, $[\iota_4,\eta_4]=[\iota_4,\iota_4]\circ\eta_7=\pm(2\nu_4\eta_7-E\nu'\eta_7)=\pm E\nu'\eta_7$, because $\nu_4\eta_7$ has order $2$. So
\[
D_8:=[\pi_5(S^4),\iota_4]=\langle E\nu'\eta_7\rangle,\qquad R:=\pi_8(S^4)/D_8\cong\mathbb{Z}/2\{\nu_4\eta_7\}.
\]

\begin{theorem}\label{thm:S4S4}
For $M=S^4\times S^4$:
\begin{enumerate}
\item[(a)] The image of $\pi\colon\Aut_\bullet(M)\to GL_2(\mathbb{Z})$ is $\{\pm I,\pm\sigma_0\}\cong(\mathbb{Z}/2)^2$.
\item[(b)] The kernel of $\pi$ is the group of pinch-fold twists, $\ker\pi\cong R\oplus R\cong(\mathbb{Z}/2)^2$, generated by $t_{j_1\nu_4\eta_7}$ and $t_{j_2\nu_4\eta_7}$.
\item[(c)] The extension splits, with section generated by the diffeomorphisms $\rho=r\times r$ and $\mathrm{sw}(x,y)=(y,x)$, where $r$ is a reflection of $S^4$. The element $\mathrm{sw}$ interchanges the two generators of $\ker\pi$, and $\rho$ acts trivially. Hence
\[
\Aut_\bullet(S^4\times S^4)\cong\bigl((\mathbb{Z}/2)^2\rtimes\mathbb{Z}/2\bigr)\times\mathbb{Z}/2\cong D_4\times\mathbb{Z}/2,
\]
a group of order $16$.
\end{enumerate}
\end{theorem}
\begin{proof}
\emph{(a)} Every $f\in\Aut_\bullet(M)$ of degree one preserves the cup product pairing, so $\pi(f)\in\Isom[\mathbb{Z}^2,s_M]$. Conversely, write $M=(S^4\vee S^4)\cup_we^8$ with $w=[\iota_1,\iota_2]$. The Whitehead product is bilinear on $\pi_4\times\pi_4$ and satisfies $[\iota_2,\iota_1]=(-1)^{4\cdot4}[\iota_1,\iota_2]=[\iota_1,\iota_2]$. For $A\in\{\pm I,\pm\sigma_0\}$ let $g_A\colon S^4\vee S^4\to S^4\vee S^4$ send $(\iota_1,\iota_2)$ to $(\iota_1,\iota_2)$, $(-\iota_1,-\iota_2)$, $(\iota_2,\iota_1)$ or $(-\iota_2,-\iota_1)$ respectively. In each case bilinearity and symmetry give $g_A\circ w=[g_A\iota_1,g_A\iota_2]=[\iota_1,\iota_2]=w$, so $g_A$ extends to a map $f_A\colon M\to M$. It acts by $A$ on $H^4$ and by $ad+bc=1$ on $H^8$, since $f_A^*(a_1a_2)=f_A^*(a_1)f_A^*(a_2)$ and $a_i^2=0$. Thus $f_A$ is a homology isomorphism of simply connected spaces, hence a based homotopy equivalence of degree one with $\pi(f_A)=A$. (For $\pm I$ and $\pm\sigma_0$ we may take the diffeomorphisms $\id$, $\rho$, $\mathrm{sw}$ and $\rho\circ\mathrm{sw}$.)

\emph{(b)} Let $f\in\ker\pi$ and write $f=(f_1,f_2)$. Then $f$ restricts to the inclusion of $S^4\vee S^4$, so $f_1|_{S^4\vee S^4}=(\iota_4,0)$ and $f_2|_{S^4\vee S^4}=(0,\iota_4)$, which are the restrictions of $\mathrm{pr}_1$ and $\mathrm{pr}_2$. By Lemma~\ref{lem:orbit}, applied to $Y=S^4$ and $g=\iota_4$ (once with $\varphi_0=g\circ\mathrm{pr}_1$ and once with $\varphi_0=g\circ\mathrm{pr}_2$), there are $b_1,b_2\in\pi_8(S^4)$ with $f_i=b_i\cdot\mathrm{pr}_i$, and both stabilizers are $[\pi_5(S^4),\iota_4]=D_8$. With $\tilde\kappa=j_1b_1+j_2b_2$ we get $f=t_{\tilde\kappa}$, as in the proof of Theorem~\ref{thm:S3S4}. Conversely each $t_{\tilde\kappa}$ lies in $\ker\pi$. So $\Psi\colon\pi_8(M)=\pi_8(S^4)^2\to\ker\pi$ is a surjective homomorphism with kernel $D_8\oplus D_8$, and $\ker\pi\cong R\oplus R$.

\emph{(c)} The diffeomorphisms $\rho$ and $\mathrm{sw}$ have degree $+1$, commute, have order $2$, and satisfy $\pi(\rho)=-I$, $\pi(\mathrm{sw})=\sigma_0$. So they generate a subgroup mapping isomorphically onto $\operatorname{im}\pi$. By Lemma~\ref{lem:coaction}(c), $\mathrm{sw}\,t_{j_1b}\,\mathrm{sw}^{-1}=t_{j_2b}$ and $\rho\,t_{j_ib}\,\rho^{-1}=t_{j_i(r\circ b)}$, and for $b=\nu_4\eta_7$ we have $r\circ b=(-\iota_4)\circ\nu_4\circ\eta_7=-\nu_4\eta_7\pm[\iota_4,\iota_4]\eta_7\equiv\nu_4\eta_7$ modulo $D_8$, since $\nu_4\eta_7$ has order $2$. Thus $\ker\pi\rtimes\langle\rho,\mathrm{sw}\rangle$ is the stated group: $\mathrm{sw}$ swaps the two summands of $R\oplus R$ and $\rho$ acts trivially. The group $(\mathbb{Z}/2)^2\rtimes\mathbb{Z}/2$ with the swap action is $\mathbb{Z}/2\wr\mathbb{Z}/2\cong D_4$.
\end{proof}

\begin{remark}
By Corollary~\ref{im}, $\gamma$ maps $\widehat\Omega^{BString}_9(P,M)$ onto $\ker\pi\cong(\mathbb{Z}/2)^2$, so this relative bordism group has at least four elements.
\end{remark}

\subsection{The products $S^p\times S^p$}
\label{subsec:SpSp}

The arguments for $S^3\times S^3$ and $S^4\times S^4$ extend to every product of two spheres of the same dimension. Let $p\ge2$ and $M=S^p\times S^p$, so $m=2p=2n$ with $n=p$. The manifold $M$ is stably parallelizable, hence a $\B$-manifold for $\B=BO\langle p+1\rangle$, it is simply connected, and $P=K(\mathbb{Z}^2,p)$ with $\Aut_\bullet(P)=GL_2(\mathbb{Z})$. Thus Theorem~\ref{thm:braid_exists} applies. Let $\iota_1,\iota_2\in\pi_p(S^p\vee S^p)$ be the two inclusions, $w=[\iota_1,\iota_2]$, and let $s_M$ be the intersection form on $H^p(M;\mathbb{Z})\cong\mathbb{Z}^2$ in the basis dual to the factors. It is $\left(\begin{smallmatrix}0&1\\1&0\end{smallmatrix}\right)$ for $p$ even and $\left(\begin{smallmatrix}0&1\\-1&0\end{smallmatrix}\right)$ for $p$ odd, so
\[
\Isom[\mathbb{Z}^2,s_M]=\{\pm I,\pm\sigma_0\}\ (p\text{ even}),\qquad \Isom[\mathbb{Z}^2,s_M]=SL_2(\mathbb{Z})\ (p\text{ odd}),
\]
the first by the computation in Section~\ref{subsec:S4xS4}. We write $[\iota_p,\iota_p]\in\pi_{2p-1}(S^p)$ for the Whitehead square. Three classical facts about it are used. For $p$ even it has infinite order, since its Hopf invariant is $\pm2$. For $p$ odd we have $2[\iota_p,\iota_p]=0$, because $[\alpha,\beta]=(-1)^{pq}[\beta,\alpha]$. Finally, by Adams \cite{adams} it vanishes if and only if $p\in\{1,3,7\}$.

\begin{theorem}\label{thm:SpSp}
Let $p\ge2$ and $M=S^p\times S^p$.
\begin{enumerate}
\item[(a)] The kernel of $\pi\colon\Aut_\bullet(M)\to GL_2(\mathbb{Z})$ is the group of pinch-fold twists, and
\[
\ker\pi\cong Q_p\oplus Q_p,\qquad Q_p:=\pi_{2p}(S^p)\big/[\pi_{p+1}(S^p),\iota_p].
\]
\item[(b)] The image of $\pi$ consists of the isometries $A=\left(\begin{smallmatrix}a&b\\c&d\end{smallmatrix}\right)$ of $s_M$ with $ab\,[\iota_p,\iota_p]=0=cd\,[\iota_p,\iota_p]$. Explicitly, it is
\begin{itemize}
\item $\{\pm I,\pm\sigma_0\}$ for $p$ even,
\item $SL_2(\mathbb{Z})$ for $p=3,7$,
\item the theta group $\Gamma_\theta=\{A\in SL_2(\mathbb{Z})\mid A\equiv I\text{ or }\sigma_0\bmod 2\}$ for $p$ odd and $p\ne1,3,7$.
\end{itemize}
\end{enumerate}
\end{theorem}
\begin{proof}
\emph{(a)} The argument of Theorem~\ref{thm:S4S4}(b) applies verbatim, with $S^4$ replaced by $S^p$ and $\pi_8$ by $\pi_{2p}$. Let $f\in\ker\pi$ and $f=(f_1,f_2)$. Since $f$ induces the identity on $\pi_p(M)=\mathbb{Z}^2$ and a map out of $S^p\vee S^p$ is determined by its values on the two spheres, $f$ restricts to the inclusion on $S^p\vee S^p$, so $f_i$ and $\mathrm{pr}_i$ agree there. By Lemma~\ref{lem:orbit} with $Y=S^p$ and $g=\iota_p$ we have $f_i=b_i\cdot\mathrm{pr}_i$ for some $b_i\in\pi_{2p}(S^p)$, and the stabilizer of $\mathrm{pr}_i$ is $[\pi_{p+1}(S^p),\iota_p]$. Put $\tilde\kappa=j_1b_1+j_2b_2\in\pi_{2p}(M)$, where $j_i$ are the inclusions of the factors. As in the proof of Theorem~\ref{thm:S3S4}, $f=t_{\tilde\kappa}$. Conversely, a twist $t_{\tilde\kappa}$ with $\tilde\kappa\in\pi_{2p}(M)$ restricts to the inclusion on $S^p\vee S^p$, has degree one because the Hurewicz image of $\tilde\kappa$ vanishes ($2p>p$), and is therefore a homotopy equivalence by Section~\ref{subsec:twists}. It lies in $\ker\pi$. The map $\Psi\colon\pi_{2p}(M)\to\ker\pi$, $\tilde\kappa\mapsto t_{\tilde\kappa}$, is a homomorphism by Lemma~\ref{lem:coaction}(b), because every class in $\pi_{2p}(M)=\pi_{2p}(S^p)^2$ is represented by a map into $S^p\vee S^p$. It is surjective, and $\Psi(\tilde\kappa)=\id$ if and only if both components $b_i$ lie in the stabilizer. Hence $\ker\Psi=[\pi_{p+1}(S^p),\iota_p]^{\oplus2}$ and $\ker\pi\cong Q_p\oplus Q_p$.

\emph{(b)} Let $f\in\Aut_\bullet(M)$. As $f$ has degree one, $\pi(f)=A$ is an isometry of $s_M$. We may take $f$ cellular. Write $K=S^p\vee S^p$ and $M=K\cup_we^{2p}$. Then $f|_K=g_A$, where $g_A$ sends $\iota_1\mapsto a\iota_1+c\iota_2$ and $\iota_2\mapsto b\iota_1+d\iota_2$. The pair $(M,K)$ is $(2p-1)$-connected and $K$ is $(p-1)$-connected, so by the Blakers--Massey theorem $\pi_{2p}(M,K)\to\pi_{2p}(M/K)=\mathbb{Z}$ is an isomorphism, as $2p\le 3p-2$. Hence the kernel of $\pi_{2p-1}(K)\to\pi_{2p-1}(M)$ is $\mathbb{Z}w$. Since $f$ is cellular, $g_A\circ w\in\mathbb{Z}w$, say $g_A\circ w=kw$, and $k$ is the degree of $f$, so $k=1$. By bilinearity of the Whitehead product on $\pi_p\times\pi_p$,
\[
g_A\circ w=[a\iota_1+c\iota_2,\,b\iota_1+d\iota_2]=ab\,[\iota_1,\iota_1]+ad\,[\iota_1,\iota_2]+cb\,[\iota_2,\iota_1]+cd\,[\iota_2,\iota_2].
\]
Now $[\iota_2,\iota_1]=(-1)^{p}[\iota_1,\iota_2]$, and $\pi_{2p-1}(K)=\pi_{2p-1}(S^p)\oplus\pi_{2p-1}(S^p)\oplus\mathbb{Z}\,w$ with $[\iota_i,\iota_i]$ the Whitehead square in the $i$-th summand. Comparing components, $g_A\circ w=w$ holds if and only if $ab[\iota_p,\iota_p]=0$, $cd[\iota_p,\iota_p]=0$ and $ad+(-1)^pbc=1$. The last condition says that $f$ has degree one. Conversely, if these three conditions hold then $g_A$ extends to a map $f_A\colon M\to M$, which acts by $A$ on $H^p$ and by $ad\pm bc=1$ on $H^{2p}$, because $x_i^2=0$ in $H^*(M)$. So $f_A$ is a homology isomorphism of simply connected spaces, hence a based homotopy equivalence of degree one with $\pi(f_A)=A$. This proves the description of the image as a set of matrices.

For $p$ even $[\iota_p,\iota_p]$ has infinite order, so the conditions say $ab=cd=0$ and $ad+bc=1$. If $a=0$ then $bc=1$ and $d=0$, and if $b=0$ then $ad=1$ and $c=0$. This gives $\{\pm I,\pm\sigma_0\}$, and these are isometries of $s_M$. For $p=3,7$ the Whitehead square vanishes and the image is all of $SL_2(\mathbb{Z})$. For the remaining odd $p$ the Whitehead square is non-zero of order $2$, and the conditions say that $ab$ and $cd$ are even. An integer matrix of determinant $1$ whose rows $(a,b)$ and $(c,d)$ both have even product is, modulo $2$, one of $I$ and $\sigma_0$. This is the theta group $\Gamma_\theta$.
\end{proof}

\begin{corollary}\label{cor:S5S5S7S7}
\begin{enumerate}
\item[(a)] $\Aut_\bullet(S^5\times S^5)\cong\Gamma_\theta$. In particular $\pi$ is injective.
\item[(b)] For $M=S^7\times S^7$ there is an exact sequence
\[
0\to(\mathbb{Z}/120)^2\to\Aut_\bullet(S^7\times S^7)\xrightarrow{\ \pi\ }SL_2(\mathbb{Z})\to1.
\]
\end{enumerate}
\end{corollary}
\begin{proof}
From Toda \cite[Chapter~V]{toda} we use $\pi_6(S^5)=\mathbb{Z}/2\{\eta_5\}$, $\pi_{10}(S^5)=\mathbb{Z}/2\{\nu_5\eta_8^2\}$ and $[\iota_5,\iota_5]=\nu_5\eta_8$, and $\pi_8(S^7)=\mathbb{Z}/2\{\eta_7\}$, $\pi_{14}(S^7)=\mathbb{Z}/120$.

(a) Since $\eta_5=E\eta_4$ is a suspension, $[\eta_5,\iota_5]=[\iota_5,\iota_5]\circ\eta_9=\nu_5\eta_8\eta_9=\nu_5\eta_8^2$, which generates $\pi_{10}(S^5)$. Hence $Q_5=0$ and $\ker\pi=0$ by Theorem~\ref{thm:SpSp}(a). The image is $\Gamma_\theta$ by part (b) of that theorem.

(b) The sphere $S^7$ is an $H$-space, so $[\iota_7,\iota_7]=0$ and $[\eta_7,\iota_7]=[\iota_7,\iota_7]\circ\eta_{13}=0$. Hence $Q_7=\pi_{14}(S^7)=\mathbb{Z}/120$, and the image is $SL_2(\mathbb{Z})$ by Theorem~\ref{thm:SpSp}.
\end{proof}

\begin{remark}
The description of $\ker\pi$ in part (a) is the case $m=n$ of Pave\v{s}i\'c's exact sequence \cite[Proposition~5.4]{pavesic}, with $Q_p=\operatorname{Coker}[\iota_p,-]$, and the determination of the image of $\pi$ by an obstruction on the attaching map $w$ is the method that Pave\v{s}i\'c attributes to Sawashita \cite{sawashita}. The point of the theorem here is that the two are put together for $S^p\times S^p$ with explicit answers. For $p=3$ and $p=4$ it recovers the kernels $(\mathbb{Z}/12)^2$ and $(\mathbb{Z}/2)^2$ found above, since $[\iota_3,\iota_3]=0$ and $Q_4=R$. For $p=7$ the kernel agrees with that in Sieradski's exact sequence \cite{sieradski}. The action of $SL_2(\mathbb{Z})$ on $(\mathbb{Z}/120)^2$ and the question whether the sequence in Corollary~\ref{cor:S5S5S7S7}(b) splits are not settled here, since the method used for $S^3\times S^3$ needs a group structure on $S^3$. By Corollary~\ref{im}, $\gamma$ maps $\widehat\Omega^{\B}_{2p+1}(P,M)$ onto $\ker\pi$, so for $p=7$ this relative string bordism group has at least $120^2$ elements.
\end{remark}

\section{Open problems}
\label{subsec:open}

\begin{enumerate}
\item \textbf{Extensions for $S^7\times S^7$ and $S^p\times S^p$.}
By Theorem~\ref{thm:SpSp} the kernel for $S^7\times S^7$ is $(\mathbb{Z}/120)^2$ and the image is $SL_2(\mathbb{Z})$. For $S^7\times S^7$ Crowley \cite{crowley} shows that the extension splits. Determine the action of $SL_2(\mathbb{Z})$ on the kernel from the braid. The argument for $S^3\times S^3$ needs a group structure on the sphere, which $S^7$ does not have, so a different approach is needed. For further $p$ the groups $Q_p$ themselves require the Whitehead squares and the composition products in $\pi_{2p}(S^p)$. It would also be of interest to compare the braid with the work of Booth and Heath \cite{booth-heath}, Heath \cite{heath} and Pave\v{s}i\'c \cite{pavesic}.

\item \textbf{Products $S^p\times S^q$ in general.}
For $M=S^p\times S^q$ with $p\ne q$, Lemma~\ref{lem:orbit} and the pinch-fold twists describe a subgroup of $\ker\pi$ as a quotient of $\pi_{p+q}(S^p)\oplus\pi_{p+q}(S^q)$ by Whitehead products. Determine when this is all of $\ker\pi$ and how $\Aut_\bullet(M)$ acts on it. The braid applies only when $q\in\{p,p+1\}$ with the Postnikov section $P$ computed as in the case of $S^3\times S^4$.

\item \textbf{The relative bordism groups.}
Compute $\widehat\Omega^{BString}_8(P,S^3\times S^4)$ and $\widehat\Omega^{BString}_9(P,S^4\times S^4)$. By Corollary~\ref{im} and Theorems~\ref{thm:S3S4} and~\ref{thm:S4S4} they have at least $48$ and at least $4$ elements, respectively, and it is not clear whether these bounds are sharp.

\item \textbf{Non-trivial fundamental groups.}
Hypothesis (T) and the proof of Lemma~\ref{lem:PD_upgrade} restrict the fundamental group to be finite or infinite cyclic. Extending the braid to manifolds with non-trivial monodromy on $\pi_i(X)$ would require tracking the $\pi_1(P)$-action through the obstruction groups, and equivariant bordism and $L$-groups would enter, as in \cite{pamuk1,pamuk2}.
\end{enumerate}

\section*{Acknowledgments}
I thank Antonio Viruel for detailed comments on an earlier version, and for pointing out the literature on self-equivalences of simply connected spaces and of products of spheres.
I used an AI language model (Claude, Anthropic) 
to edit the manuscript’s language, check LaTeX, and review selected mathematical arguments.  I have verified all of the mathematics and take full responsibility for the content.  

\bibliographystyle{amsplain}
\providecommand{\bysame}{\leavevmode\hbox to3em{\hrulefill}\thinspace}
\providecommand{\MR}{\relax\ifhmode\unskip\space\fi MR }
\providecommand{\MRhref}[2]{\href{http://www.ams.org/mathscinet-getitem?mr=#1}{#2}}
\providecommand{\href}[2]{#2}

\end{document}